\documentclass[a4paper, 12pt]{article}
\usepackage{amsmath}
\usepackage{amssymb}
\usepackage{amscd}
\usepackage{amsthm}
\usepackage[totalwidth=16.5cm, totalheight=24.5cm]{geometry}
\usepackage{graphicx}
\usepackage{hyperref}    
\usepackage{enumerate}
\usepackage{picins}

\newtheoremstyle{boldplain}
{9pt}
{9pt}
{\itshape}
{}
{\bfseries}
{.}
{.5em}
{\thmname{#1}\thmnumber{ #2}\thmnote{ (#3)}}%

\newtheoremstyle{bolddefinition}
{9pt}
{9pt}
{}
{}
{\bfseries}
{.}
{.5em}
{\thmname{#1}\thmnumber{ #2}\thmnote{ (#3)}}%

\theoremstyle{boldplain}

\newtheorem{corollary}[equation]{Corollary}
\newtheorem{lemma}[equation]{Lemma}
\newtheorem{proposition}[equation]{Proposition}
\newtheorem{question}[equation]{Question}

\newtheorem{theorem}[equation]{Theorem}
\newtheorem{theoremi}{Theorem}

\newtheorem{corollaryi}[theoremi]{Corollary}

\theoremstyle{bolddefinition}

\newtheorem{example}[equation]{Example}
\newtheorem{remark}[equation]{Remark}

\numberwithin{equation}{section}

\newcommand{\R}{{\mathbb R}}
\newcommand{\Z}{{\mathbb Z}}
\newcommand{\N}{{\mathbb N}}

\newcommand{\acts}{\curvearrowright}

\newcommand{\ora}{\overrightarrow}
\newcommand{\pihalf}{\frac{\pi}{2}}
\newcommand{\tits}{\partial_{T}}

\newcommand{\Si}{\Sigma}

\DeclareMathOperator{\rad}{rad}
\DeclareMathOperator{\inrad}{inrad}
\DeclareMathOperator{\CAT}{CAT}

\begin{document}

\title{Subcomplexes and fixed point sets of isometries of spherical buildings}
\author{Carlos Ramos-Cuevas}

\maketitle

\begin{abstract}
In this paper we study convex subcomplexes of spherical buildings. 
We pay special attention to fixed point sets of type-preserving
isometries of spherical buildings. This sets are also convex subcomplexes
of the natural polyhedral structure of the building.
We show, 
among other things,
that if the fixed point set is top-dimensional
then it is either a subbuilding or it has circumradius $\leq \pihalf$.
If the building is of type $A_n$ or $D_n$, 
we also show that the same conclusion holds for an arbitrary
(top-dimensional in the $D_n$-case) convex subcomplex.
This proves a conjecture of Kleiner-Leeb \cite[Question 1.5]{KleinerLeeb:invconvex}
in these cases.
\end{abstract}

\section{Introduction}\label{sec:intro}

We are interested in the following geometric
question about convex subsets of spherical buildings.

\begin{question}\label{ques:main}
 Let $C$ be a convex subset of a spherical building $B$. 
Is it true that $C$ is either a subbuilding 
or it has circumradius $\leq \pihalf$ (i.e.\ it is contained in a ball of radius $\leq \pihalf$
centered in $C$)?
\end{question}

This question was first asked by Kleiner and Leeb \cite[Question 1.5]{KleinerLeeb:invconvex}
while studying rigidity properties of convex subsets of symmetric spaces of higher rank.
A closely related (and weaker) question is Tits' Center Conjecture,
it is concerned with convex subsets, which are also subcomplexes of the natural polyhedral
structure of a spherical building, and fixed points of their automorphism groups.
We refer to \cite{LeebRamos-Cuevas:centerconj} for more information on the 
Center Conjecture and its relationship with Question~\ref{ques:main}.
This conjecture has been recently proven
(see \cite{MuehlherrTits:centerconj}, \cite{LeebRamos-Cuevas:centerconj}, 
\cite{Ramos-Cuevas:centerconj}, \cite{MuehlherrWeiss:centerconj}).
This recent success with the Center Conjecture suggests that 
there might be better prospects for getting 
an answer to Question~\ref{ques:main} if we restrict our
attention to convex {\em subcomplexes}.

The first two of the main results in this paper investigate general convex subcomplexes.
If the building is of type $A_n$, then we can answer Question~\ref{ques:main} for any convex
subcomplex.

\begin{theoremi}\label{thmintro:An}
Let $B$ be a (not necessarily thick) spherical building of type $A_n$. 
Let $C\subset B$ be a convex subcomplex.
Then either $C$ is a subbuilding or it has circumradius $\leq \pihalf$. 
\end{theoremi}

If the building is of type $D_n$, then we can treat the case of top-dimensional 
convex subcomplexes.

\begin{theoremi}\label{thmintro:Dn}
Let $B$ be a (not necessarily thick) spherical building of type $D_n$. 
Let $C\subset B$ be a top-dimensional convex subcomplex.
Then either $C$ is a subbuilding or it has circumradius $\leq \pihalf$. 
\end{theoremi}

A prominent example (and Tits' first 
motivation for the Center Conjecture, cf.\ \cite{Tits:firstcenterconj}) 
of convex subcomplexes in spherical buildings 
are fixed point sets of type-preserving isometries.
We want to focus now our attention on these kind of convex subsets.

Our first main theorem about fixed point sets of isometries is a positive
answer to Question~\ref{ques:main} in the top-dimensional case.

\begin{theoremi}\label{thmintro:topdim}
Let $B$ be a spherical building and let $g\in Isom(B)$ be an isometry 
whose fixed point set $Fix(g)\subset B$ is top-dimensional. 
Then either $Fix(g)$ is a subbuilding
or it has circumradius $\leq \pihalf$. 
\end{theoremi}

For fixed point sets of groups of isometries we have the following immediate corollary 
(Corollary~\ref{cor:fixedptssubgp}).

\begin{corollaryi}
 Let $H\subset Isom(B)$ be a subgroup of isometries such that 
the fixed point set $Fix(H)$ is top-dimensional. Suppose that there is an element $g\in H$
such that $Fix(g)$ is not a subbuilding. Then $Fix(H)$ has circumradius $\leq \pihalf$. 
\end{corollaryi}

If the building has no factors of exceptional type, then we can drop the assumption
of top-dimensional fixed point set.

\begin{theoremi}\label{thmintro:poscodim}
Let $B$ be a spherical building without factors of 
type $F_4,E_6,E_7, E_8$ and let $g\in Isom(B)$ be a type-preserving isometry. 
Then either $Fix(g)$ is a subbuilding
or it has circumradius $\leq \pihalf$. 
\end{theoremi}

A main step in the proof of Theorems~\ref{thmintro:topdim} and \ref{thmintro:poscodim}
is to show first the same result for a special kind of isometries with top-dimensional
fixed point sets, namely the unipotent isometries (see Theorem~\ref{thm:unipotentisom}).

As a motivation, 
let us consider first the example of a spherical building $B=\tits X$ which is the Tits boundary
of a symmetric space $X=G/K$ of noncompact type.
An isometry $g\in Isom_0(X)\cong G$ induces an isometry $g_T$ of the building $B$.
If the isometry $g$ is semisimple, then its fixed point set $Fix(g)\subset X$ is a totally geodesic
subspace and its boundary at infinity $\tits Fix(g) = Fix(g_T)$ is a subbuilding of $B$.
If $g$ is parabolic, then $Fix(g_T)$ has circumradius $\leq \pihalf$
(see \cite[Prop.\ 4.1.1]{Eberlein:nonpositive}, \cite[Lemma 3]{BallmannGromovSchroeder}). 
Hence, we obtain 
a positive answer to Question~\ref{ques:main} in this case. 
The fact above about the circumradius of the fixed point set at infinity of a parabolic isometry
holds in more generality, e.g.\ for any $\CAT(0)$ space $X$ of finite dimension,
this is shown in \cite{LytchakCaprace:atinfinity} (see also \cite{FujiwaraNaganoShioya}). 
The proof of this result goes roughly as follows:
consider the displacement function $d_g(x)=d(x,gx)$, $x\in X$ of the parabolic isometry $g$. 
This function is convex and Lipschitz. 
Now we follow in $X$ a path  in the direction of the greatest decrease of the function $d_g$
(e.g.\ if $X$ is a Riemannian manifold, then we just follow a flow line of minus the gradient
of $d_g$). Since $g$ is parabolic, $d_g$ does not attain its infimum in $X$ 
and this path must have an accumulation
point $\xi\in \tits X$ at infinity. One then shows that 
$\xi\in Fix(g_T)$ and
for all $\zeta\in Fix(g_T)$ holds $d(\xi,\zeta)\leq \pihalf$.
Now suppose that the spherical building 
$B=\tits X$ is the Tits boundary of a Euclidean building.
An isometry $g\in Isom(X)$ induces again an isometry of $B$.  
But in contrast to the symmetric space case, a Euclidean building admits no
parabolic isometries (see \cite{Parreau}, \cite{Ramos-Cuevas:displacement}).
Further, the fixed point sets at infinity of semisimple  isometries
are not necessarily subbuildings. 
Hence, we cannot apply the result above to give an answer to 
Question~\ref{ques:main} in this case.

We will nevertheless rescue the main idea in the proof mentioned above, that is, to follow
the direction of the greatest decrease of a convex function.
For this purpose we forget about 
the $\CAT(0)$ space $X$ and work directly with convex functions 
defined in a convex subset $C$ of the building $B$.
For a special family of convex functions, 
which we call {\em nicely convex} in Section~\ref{sec:convexfunction},
we can assure the existence of a unique point
$x\in C$ where the function attains its minimum and for this point holds
$d(x,y)\leq\pihalf$ for all $y\in C$ (Lemma~\ref{lem:radiusminimum}).
Thus, the main work will go into finding such functions for the convex subsets in question.
Actually, a positive answer to Question~\ref{ques:main} is equivalent to the existence
of such convex functions for convex subsets which are not subbuildings
(see Proposition~\ref{prop:equivconj}).
This idea is inspired by the approach to the Center Conjecture initiated in
\cite{BateMartinRohrle:centerconj} using Geometric Invariant Theory.

The functions that we will consider measure essentially 
the negative of the distance of a point to the boundary
of the convex subset (Section~\ref{sec:weightedincenter}). 
It is easy to see that in the case of convex subsets of spheres such
functions satisfy the desired conditions for nice convexity.
However, in general, these functions will not be even convex.
Our strategy will be as follows,
for a given convex subcomplex we find a family of apartments,
which is big enough such that any pair of points of the subcomplex is contained
in an apartment of the family, and small enough such that the value 
of the function above for a given apartment of this family
at a point of the subcomplex
 does not depend on the apartment containing the point.
This will define a convex function on the convex subcomplex which is nicely convex
since its restriction to any apartment of the family (which is just a sphere) is nicely convex. 
For instance, in the case of fixed point sets of unipotent isometries we will
see that we can take the collection of all apartments of the building
(Theorem~\ref{thm:unipotentisom}).

Independently of our interest in Question~\ref{ques:main} another motivation to
study fixed point sets is to investigate the relationship between algebraic properties of
an isometry $g\in Isom(B)$ and the geometry of its fixed point set $Fix(g)\subset B$.
This will be the main subject in Section~\ref{sec:applications}.
For instance, we give a geometric proof of the well known fact (for algebraic groups)
that the product of two commuting unipotent elements is again unipotent 
(see Proposition~\ref{prop:commutingunipotent}).
Another question in this direction is to what extent we can read off the fixed point set 
the Jordan decomposition of an element of an algebraic group. 
The Jordan decomposition will be discussed in Section~\ref{sec:jordan}.

\tableofcontents

\section{Preliminaries}

In this paper we consider spherical buildings from the $\CAT(1)$ 
viewpoint as presented in \cite[Section 3]{KleinerLeeb:quasi-isom},
we refer to it and \cite{Tits:bn-pairs, Weiss:sphbuild, AbramenkoBrown} 
for the basic definitions and facts about 
spherical Coxeter complexes and spherical buildings.
For more information on $\CAT(1)$ spaces in general we refer to \cite{BridsonHaefliger}.

\subsection{Spherical joins}

Consider the planes 
$E_i:=\{x\in \R^{2n}\;|\; x_k=0 \text{ if } k\neq 2i-1,2i\}$ in $\R^{2n}$
for $i=1,\dots, n$, and the corresponding
unit circles $S^1_i:=S^{2n-1}\cap E_i$.
For a point $y$ in the unit round sphere $S^{2n-1}\subset \R^{2n}$, 
let $a_i\geq 0$ be the norm of the projection 
$\pi_i(y)$ to the plane $E_i$ and let $y^i:=\pi_i(y)/a_i\in S^1_i$ if $a_i\neq 0$,
then $y=\sum_i a_i y^i$. Hence $S^{2n-1}$ can be thought of as a kind
of product of $n$ copies of the unit circle $S^1$ induced by the 
metric product $\R^n \cong E_1\times \dots \times E_n$.
The {\em spherical join}
generalizes this construction to metric spaces.

Let $S^n_+ \subset S^n\subset \R^{n+1}$ denote the points on the round unit 
$n$-sphere $S^n$ with all their coordinates non-negative.
Let $Y_1,\dots Y_n$ be metric spaces of diameter $\leq \pi$.
The {\em spherical join} $Y=Y_1\circ\dots\circ Y_n$ is the metric space, that
as a set is $S^n_+\times Y_1\times\dots\times Y_n$ modulo the equivalence relation
that identifies $(a,y_1,\dots,y_n)\sim (a',y_1',\dots,y_n')$ if and only if
$a=a'\in S^n_+$, and $y_i=y_i'$ whenever $a_i=a_i'\neq 0$.
There is a natural identification (as sets) of $S^{2n-1}$ and the spherical join
of $n$ copies of $S^1$ (see above). We define the metric on $S^1\circ\dots\circ S^1$
such that this identification is an isometry and use this metric to define the metric
of general spherical joins. 
Let $y=(a,y_1,\dots,y_n),y'=(a',y_1',\dots,y_n')$ be two
points in $Y=Y_1\circ\dots\circ Y_n$. Choose points $p_i,p_i'\in S^1$
such that $d(y_i,y_i')=d(p_i,p_i')$. Then we define the distance between $y,y'$
as the distance between the points 
$(a,p_1,\dots,p_n),(a',p_1',\dots,p_n')\in S^1\circ\dots\circ S^1\cong S^{2n-1}$.

There are natural isometric embeddings $Y_i\hookrightarrow Y$. 
Thus, we may think of each $Y_i$
as a subspace of $Y$.

Notice that this definition is made 
{\em ad-hoc} such that the Euclidean cone over $Y$ is 
canonically isometric to the product of the Euclidean cones over the $Y_i$.

The next lemma follows directly from the definition.

\begin{lemma}\label{lemma:join}
Let $Y$ be a metric space of diameter $\leq \pi$. Then the diagonal map 
$Y\rightarrow Y\circ \dots \circ Y$ given by
$y\mapsto [((\frac{1}{\sqrt{n}},\dots,\frac{1}{\sqrt{n}}),(y,\dots,y))]$ 
is an isometric embedding. \qed
\end{lemma}

\subsection{Circumradius and inradius}

Let $y$ be a point in a metric space $Y$. 
The {\em circumradius of $Y$ with respect to $y$} is defined as 
$\rad(Y,y):=\sup_{x\in Y}d(x,y)$ and the {\em circumradius} of $Y$ is 
$\rad(Y):=\inf_{y\in Y} \rad(Y,y)$,
that is, $\rad(Y)$ is the infimum of the radii of balls centered at $Y$ and containing it.
A point $y\in Y$, such that $\rad(Y,y)=\rad(Y)$ is called a {\em circumcenter}.

Let $C\subset Y$ be a subset. 
The {\em inradius of $C$ with respect to $y\in C$} is defined as 
$\inrad(C,y):=\sup\{r\geq 0 \;|\; B_r(y)\subset C\}$ and the {\em inradius} of $C$ is 
$\inrad(C):=\sup_{y\in C} \inrad(C,y)$,
that is, $\inrad(C)$ is the supremum of the radii of balls contained in $C$.
A point $y\in Y$, such that $\inrad(Y,y)=\inrad(Y)$ is called an {\em incenter}.

\subsection{$\CAT(1)$ spaces}

A metric space is called a {\em $\CAT(1)$ space} 
if it is $\pi$-geodesic 
and geodesic triangles of perimeter less than $2\pi$
are not {\em thicker} than those in the round unit sphere.

For points $x,y$ in a $\CAT(1)$ space $Y$ at distance $<\pi$, 
we denote by $xy$ the unique segment connecting both points. 
Two points at distance $\geq\pi$ are called {\em antipodal}.
The {\em link} $\Si_y Y$ at a point $y\in Y$ is the space of directions at $y$
with the angle metric. It is again a $\CAT(1)$ space. 
 If $y\neq x$  and $y$ is not antipodal to $x$, we denote
 with $\ora{xy}\in\Si_x X$ the direction at $x$ of the segment $xy$. 

A subset $C$ of a $\CAT(1)$ space is called {\em convex}, 
if for any 
$x,y\in C$ at distance $<\pi$ the segment $xy$ 
is also contained in $C$.
A convex subset of a $\CAT(1)$ space is itself $\CAT(1)$.
The {\em convex hull} $CH(A)$ of a subset $A\subset Y$ 
is the smallest closed convex subset of $Y$ containing $A$.

\subsubsection{Convex functions on $\CAT(1)$ spaces}\label{sec:convexfunction}

Let $Y$ be a $\CAT(1)$ space. A function $f:Y\rightarrow \R$ is said to
be {\em (strictly) convex} if for any geodesic segment $c:[0,L]\rightarrow Y$ of
length $L<\pi$, the function $f\circ c :[0,L]\rightarrow\R$ is (strictly) convex.

Suppose that $f:Y\rightarrow \R$ is a nonpositive function such that
$$ f(x) + f(y) \geq 2\cos(\tfrac{d(x,y)}{2})f(m) $$
for all $x,y \in Y$ at distance $<\pi$, where $m\in Y$ is the midpoint of 
the segment $xy$. Notice that since $f(m)\leq 0$, the inequality already implies
that $f$ is convex and also that $f$ is strictly convex on the convex subset $\{f < 0\}$.
Suppose further that $\{f < 0\}$ does not contain pairs of 
antipodal points. We say that a function with these properties is 
{\em nicely convex}.

\begin{remark}\label{rem:slope}
 Our motivation to consider this special kind of convex functions is the following.
Suppose that $Y'$ is the Tits boundary of a $\CAT(0)$ space $X$. Let 
$h:X\rightarrow\R$ be a continuous convex function. 
Let $\rho$ be a geodesic ray in $X$ with $\rho(\infty)=\xi$.
Then $slope_h(\xi):= \lim_{t\to\infty}\frac{h(\rho(t))}{t}$ defines a function on $Y'$,
whose restriction to the convex subset
$Y=\{slope_h\leq 0\}\subset Y'$ satisfies the properties above.
(cf.\ \cite[Section 3.1]{KapLeebMillson:convex}, 
\cite[Section 4.2]{LytchakCaprace:atinfinity}.)
\end{remark}

Let $f:Y\rightarrow \R$ be nicely convex. If $y_n\in Y$ is a sequence with
$f(y_n)\rightarrow \inf f < 0$, then the inequality implies that $(y_n)$ is Cauchy. 
Hence, if $Y$ is complete, then $f$ attains its infimum at a unique point 
$y_f\in \{f < 0\}$.

\begin{lemma}\label{lem:radiusminimum}
Let $Y$ be a complete connected $\CAT(1)$ space and 
let $f:Y\rightarrow \R$ be a nonconstant nicely convex function.
Then for the unique minimum $y_f\in Y$ of $f$ holds
$\rad(Y,y_f)=\sup_{y\in Y}d(y,y_f) \leq \pihalf$.
\end{lemma}

\begin{proof}
Let $y\in Y$ with $d(y,y_f) <\pi$ and
let $c:[0,L]\rightarrow \R$ be the geodesic segment with $c(0)=y_f$
and $c(L)=y$. Write $\bar f:= f\circ c $. The inequality for nice convexity
implies $\bar f(t) + \bar f(0)\geq 2\cos(\frac{t}{2})\bar f(\frac{t}{2})$
 for all $t\in[0,L]$. 

Let $g_0(t):=1$ and $g_{n+1}(t):=2\cos(\frac{t}{2})g_{n}(\frac{t}{2})-1$.
We show inductively that $\bar f(t)\geq \bar f(0)g_n(t)$ for all $n\geq 0$
and $t\in[0,L]$: For $n=0$ this is clear since $\bar f(0)$ is the minimum of $f$.
By induction we have
$\bar f(t) + \bar f(0)\geq 2\cos(\frac{t}{2})\bar f(\frac{t}{2})\geq 
2\cos(\frac{t}{2})\bar f(0)g_{n-1}(\frac{t}{2})$ and this in turn implies that
$\bar f(t) \geq \bar f(0) (2\cos(\frac{t}{2})g_{n-1}(\frac{t}{2})-1 ) = \bar f(0)g_n(t)$.

Now we claim that $g_n(t)\xrightarrow[n\rightarrow \infty]{} \cos(t)$. Indeed, let us 
compute 
\begin{align*}
 |\cos t - g_n(t)| &= |(2\cos^2(\tfrac{t}{2})-1)-(2\cos(\tfrac{t}{2})g_{n-1}(\tfrac{t}{2})-1)|
 =2\cos(\tfrac{t}{2})|\cos(\tfrac{t}{2}) - g_{n-1}(\tfrac{t}{2})|=\dots \\
&= 2^n\cos\tfrac{t}{2}\cos\tfrac{t}{4}\dots\cos(\tfrac{t}{2^{n-1}})(1-\cos\tfrac{t}{2^{n}})
\leq 2^n(1-\cos\tfrac{t}{2^{n}}). 
\end{align*}
But $2^n(1-\cos\tfrac{t}{2^{n}})\rightarrow 0$ as can be seen e.g.\ with L'H\^{o}pital.

Therefore we obtain $0\geq \bar f(L)\geq f(0)\cos L$. Hence, $\cos L \geq 0$ and we
conclude that $L\leq \pihalf$.
Thus, for any $y\in Y$ holds $d(y,y_f) \leq \pihalf$ or $d(y,y_f) \geq \pi$. 
Since $Y$ is connected, we conclude that $\sup_{y\in Y}d(y,y_f) \leq \pihalf$.
\end{proof}

\begin{remark}
Let $Y$ be a complete $\CAT(1)$ space and $C\subset Y$ a connected, closed, convex subset.
Lemma~\ref{lem:radiusminimum} implies that if $C$ admits a 
nonconstant nicely convex function, then 
$\rad(C)=\inf_{x\in C} \rad(C,x) \leq \pihalf$. Conversely, if $\rad(C)\leq \pihalf$
and $C$ has finite dimension, then by \cite[Lemma 3.3]{BalserLytchak:centerbuild}
$C$ has a circumcenter $x_0$, that is, $\rad(C,x_0) = \rad(C) \leq \pihalf$.
The function $-\cos(d(\cdot,x_0))$ is nicely convex in $C$ by Lemma~\ref{lem:distbound}
and triangle comparison. 
\end{remark}

This observation applied to convex subsets of spherical buildings 
generalizes 
\cite[Thm.\ 4.5]{BateMartinRohrle:centerconj}:

\begin{proposition}\label{prop:equivconj}
 A closed convex subset of positive dimension of a spherical building has circumradius
at most $\pihalf$ if and only if it admits a nonconstant nicely convex function. \qed
\end{proposition}

Now we see a special example of a nicely convex function on a sphere, 
that we will use later.

\begin{lemma}\label{lem:distbound}
 Let $H\subset S^n\subset \R^{n+1}$ be a hemisphere of the round unit sphere $S^n$.
The function $f:H\rightarrow \R$ given by $f(x)=-\sin(d(x,\partial H))$ is nicely convex.
\end{lemma}
\begin{proof}
 Let $x_0$ be the center of the hemisphere $H$ and let
$\langle,\rangle$ denote the standard scalar product in $\R^{n+1}$.
Then 
$f(x)=-\sin d(x,\partial H) = -\cos d(x,x_0) =-\langle x,x_0\rangle$. 
Let $x,y\in H$ with $d(x,y)<\pi$. Then
$f(x)+f(y)=-\langle x+y , x_0\rangle = -\Vert x+y\Vert\langle 
\frac{x+y}{\Vert x+y\Vert}, x_0\rangle =\Vert x+y\Vert f(m) = 2\cos\frac{d(x,y)}{2}f(m)$, 
where $m\in H$ is the midpoint of the segment $xy\subset S^n$.
And since the interior of $H$ does not contain antipodal points, it follows that $f$ is 
nicely convex.
Alternatively, the lemma follows from Remark~\ref{rem:slope} and the fact that $f$ is the
slope of the convex function $-\langle \cdot,x_0\rangle$ in $\R^{n+1}$.
\end{proof}

\begin{corollary}[][cf.\ {\cite[Lemma 1]{White:capsphere}}]\label{cor:incenter}
 Let $C\subsetneq S^n\subset \R^{n+1}$ be a closed convex subset
with non-empty interior. 
Then $C$ has a unique
incenter. That is, there is a unique $x_0\in C$ such that 
$d(x_0,\partial C) = \sup_{x\in C} d(x,\partial C)$. Moreover, $\rad(C,x_0)\leq \pihalf$.
\end{corollary}
\begin{proof}
 The function $-\sin(d(\cdot, \partial C))=
\sup_{x\in C} \{-\sin(d(\cdot, \partial H_x))\}$, where $H_x$ is a hemisphere
with $x\in \partial H_x$ and $C\subset H_x$, is nicely convex by the previous lemma.
The unique minimum of this function is the incenter of $C$.
\end{proof}

\begin{remark}
 The conclusion of Corollary~\ref{cor:incenter} is not true anymore for convex
subsets of spherical buildings as we will see later. In general, the function
$-\sin(d(\cdot, \partial C))$ is not even convex.
\end{remark}

\subsection{Spherical Coxeter complexes}

A spherical Coxeter complex $(S,W)$ is a pair consisting in
a unit round sphere $S=S^n\subset \R^{n+1}$
together with a finite group of isometries $W$, called
the {\em Weyl group}, generated by linear reflections at hyperplanes.

The spheres of codimension one in $S$, 
that are the fixed point sets of the reflections in $W$ 
are called the {\em walls}.
The {\em Weyl chambers} or just {\em chambers}
are the closures of the connected components of 
$S$ minus the union of all the walls.
A Weyl chamber is a convex spherical polyhedron, 
they are fundamental domains for the action
of the Weyl group on $S$ and therefore isometric to the
{\em model Weyl chamber} $\triangle_{mod}:=S/W$.
A {\em root} is a top-dimensional hemisphere bounded by a wall.
A {\em singular} sphere is an intersection of walls.
A {\em face} is the intersection of a Weyl chamber and a singular sphere.
The codimension one faces
of a Weyl chamber are called {\em panels}.
The center of a root is called a point of {\em root-type}.

The geometry of a spherical Coxeter complex can be encoded in a graph, the so-called
{\em Dynkin diagram}. We say that $(S,W)$ is of {\em simply-laced type} if its 
Dynkin diagram has no loops, that is, if its irreducible factors are of type $A_n,D_n,E_n$.
 A labelling by an index set $I$ of the vertices of the Dynkin diagram induces a labelling of
 the vertices of the model Weyl chamber $\triangle_{mod}$. 
We say that a vertex in $S$ is of  {\em type} $i$ or that it is an $i$-vertex for $i\in I$,
 if its projection under $S\rightarrow S/W = \triangle_{mod}$ has label $i$.

Suppose $(S,W)$ is irreducible (i.e.\ its Dynkin diagram is connected).
If $(S,W)$ is simply-laced, then there is only one $W$-orbit of points of root-type.
Their possible mutual distances are $0,\frac{\pi}{3},\pihalf,\frac{2\pi}{3},\pi$.
If $(S,W)$ is not simply-laced, then there are two $W$-orbits of root-type points.
Root-type points in different orbits have mutual distances $\frac{\pi}{4},\pihalf,\frac{3\pi}{4}$.
If $(S,W)$ is of type $B_n$, then the possible mutual distances between root-type
points in one of the $W$-orbits are $0,\frac{\pi}{3},\pihalf,\frac{2\pi}{3},\pi$, and in the other orbit
are $0,\pihalf,\pi$. If $(S,W)$ is of type $F_4$, then root-type points in the same orbits have mutual
distances $0,\frac{\pi}{3},\pihalf,\frac{2\pi}{3},\pi$.
For more information on possible distances between vertices in a spherical Coxeter complex
we refer to \cite[Section 2.2]{LeebRamos-Cuevas:centerconj} and 
\cite[Section 3]{Ramos-Cuevas:centerconj}.

\subsubsection{Root systems}\label{sec:rootsystem}

Let $\Phi$ be an irreducible root system in $\R^{k}$ (we refer to 
\cite[Chapter VI]{Bourbaki:lie} for the definition). 
A root $\alpha\in \Phi$ is called {\em reduced} if $2\alpha\notin \Phi$. A root system, 
whose roots are all reduced is also called reduced.
The reduced root systems of rank $k\geq3$
are of type $A_k, B_k, C_k, D_k, E_6,E_7, E_8, F_4$.
The non-reduced root systems are of type $BC_k$.
A root $\alpha\in \Phi$ is called {\em divisible} if $\alpha/2\in\Phi$.
Notice that for a root, non-reduced implies indivisible. 
We say that a root 
$\alpha\in \Phi$ is a {\em short} root if its length is minimal among roots in $\Phi$
and a {\em long} root if its length is maximal. 
If $\Phi$ is non-reduced (i.e.\ of type $BC_k$), then
for a short root $\alpha$ it holds that $\alpha$ is non-reduced and
$2\alpha$ is a long root.  

We suppose that $\Phi$ is always so normalized that a short root
$\alpha\in \Phi$ has norm $\Vert \alpha\Vert =1$.
In particular, the possible norms of roots are $\{1,\sqrt{2},2\}$.

The reflections on hyperplanes orthogonal to roots in $\Phi$ generate a finite
group of isometries $W$ of $R^{k}$. If we restrict the action of $W$ to the unit sphere
$S=S^{k-1}\subset\R^k$ we obtain a spherical Coxeter complex $(S,W)$. 

Let $\alpha \in \Phi$ be an indivisible root. We denote also with $\alpha\subset S$ the root 
(i.e.\ the hemisphere) $\{x\in S \;|\; \langle \alpha, \cdot\rangle \geq 0\}\subset S$.
Conversely, if $\alpha\subset S$ is a root, then we denote again with $\alpha\in\Phi$ the 
corresponding indivisible root. There should be no confusion with this abuse of notation. 

\subsubsection{Convex subcomplexes of spherical Coxeter complexes}\label{sec:subcomplex}

Let $K\subset S$ be a convex subcomplex, that is, $K$ is an intersection of roots in $S$.
Let $s\subset S$ be the singular sphere of the same dimension as $K$ containing $K$.
We define $\Lambda_K$ to be the set of the singular hemispheres $h\subset s$ containing $K$.
Then $\Lambda_K$ is the largest set of singular hemispheres such that 
$K=\bigcap_{h\in\Lambda_K}h$.
Let $\Lambda_K^{min}\subset\Lambda_K$ be the set of singular hemispheres $h\in\Lambda_K$ such that 
$(-h)\cap K$ has codimension one in $s$, where $-h$ is the other hemisphere in $s$ with 
$\partial(-h)=\partial h$.
That is, each $h\in\Lambda_K^{min}$ determines a boundary component in $\partial K$
of codimension one, hence, 
$\Lambda_K^{min}$ is the minimal set of singular hemispheres in $s$
such that $K=\bigcap_{h\in\Lambda_K^{min}}h$.
Notice that if $K$ is top-dimensional, then $\Lambda_K$ is a set of roots in $S$.

\subsubsection{Weighted incenter of a top-dimensional subcomplex}
\label{sec:weightedincenter}

Let $\Phi$ be a (not necessarily irreducible) root system in $\R^{n+1}$
and let $(S=S^n,W)$ be its associated spherical Coxeter complex.
We choose weights $\mu_\alpha>0$ for each root $\alpha\subset S$ as follows. 
If the corresponding indivisible root $\alpha\in \Phi$ is reduced, then
we set $\mu_\alpha = \Vert \alpha\Vert$.
If $\alpha\in \Phi$ is non-reduced, then we can choose 
$\mu_\alpha\in\{1,2\}$.
Notice that we only have a flexibility on the choice of the weights $\mu_\alpha$
if the root system $\Phi$ is non-reduced.

For a root $\alpha\subset S$, let $x_\alpha\in S$ denote the center of $\alpha$.
If the corresponding indivisible root $\alpha\in \Phi$ is reduced, then
$\mu_\alpha x_\alpha=\alpha \in \Phi$; if it is non-reduced, then
$\mu_\alpha x_\alpha \in\{\alpha,2\alpha\} \subset \Phi$.

Let $K\subset S$ be a proper top-dimensional convex subcomplex.
We define the function $f_K$ in $K$ as 
$$f_K(x):= \max_{\alpha\in\Lambda_K}\{-\mu_\alpha\sin( d(x,\partial\alpha))\}.$$

The function $f_K$ is nicely convex by Lemma~\ref{lem:distbound} and therefore has a unique
minimum, which we call the {\em weighted incenter} of $K$. Notice that if $(S,W)$ is
of simply-laced type, then $\mu_\alpha=1$ for all roots and 
the weighted incenter is the same as the incenter of $K$.

Notice that 
$\mu_\alpha\sin(d(x,\partial\alpha))
=\mu_\alpha\cos(d(x,x_\alpha))=\langle x,\mu_\alpha x_\alpha\rangle$.
Hence, the function $f_K(x)$ is given by 
$\max_{\alpha\in\Lambda_K}\{-\langle x,\mu_\alpha x_\alpha\rangle\}$.

The next lemma shows that we can define the function $f_K$ using the smaller
set of roots $\Lambda_K^{min}$ (or any set of roots between $\Lambda_K$
and $\Lambda_K^{min}$). 

\begin{lemma}\label{lem:weightedinrad}
With the notation above,
$f_K(x)=\max_{\alpha\in\Lambda_K^{min}}\{-\mu_\alpha\sin( d(x,\partial\alpha))\}$.
\end{lemma}
\begin{proof}
Clearly $f_K(x)\geq f_K^{min}(x):=
\max_{\alpha\in\Lambda_K^{min}}\{-\mu_\alpha\sin( d(x,\partial\alpha))\}$.
For any root 
$\beta\in\Lambda_K-\Lambda_K^{min}$ holds that $d(x,\partial\beta) = \pihalf$;
or, if $d(x,\partial\beta)<\pihalf$, then the 
segment between $x$ and its projection to $\partial\beta$ must cross the boundary of $K$,
in particular, it must intersect a wall $\partial\alpha_0$ for some $\alpha_0\in \Lambda_K^{min}$.
The next Lemma implies 
$-\mu_\beta\sin d(x,\partial\beta) \leq -\mu_{\alpha_0}\sin d(x,\partial\alpha_0)$
and therefore $f_K(x)\leq f_K^{min}(x)$.
\end{proof}

\begin{lemma}\label{lem:weightedinrad2}
 Let $x\in S$ and let $\alpha,\beta\subset S$ be two roots containing $x$ such that
$d(x,\partial\beta)=\pihalf$; or, $d(x,\partial\beta)<\pihalf$ 
and the segment between $x$ and its projection to $\partial \beta$ intersects $\partial\alpha$.
Then $\mu_\beta\sin d(x,\partial\beta) \geq \mu_\alpha\sin d(x,\partial\alpha)$.
\end{lemma}
\begin{proof}
It follows from the conditions that $d(x,\partial\beta)\geq d(x,\partial\alpha)$.
If $\mu_\beta \geq \mu_{\alpha}$, then the assertion follows.
So suppose that
$\mu_\beta < \mu_{\alpha}$. In particular, $(S,W)$ is not of simply-laced type.
If its root system $\Phi$ is reduced this implies that $(\mu_{\alpha},\mu_\beta)=(\sqrt{2},1)$ and
if $\Phi$ is non-reduced, then $(\mu_{\alpha},\mu_\beta)=(2,\sqrt{2}),(2,1),(\sqrt{2},1)$.

Notice that since the segment between 
$x$ and its projection to $\partial \beta$ intersects $\partial\alpha$
we must have 
$d(x,\alpha\cap\partial\beta)\geq d(x,(-\alpha)\cap\partial\beta)$.
Let $C:= \{y\in\beta\;|\;d(y,\alpha\cap\partial\beta)\geq d(y,(-\alpha)\cap\partial\beta)\}$. 
Then $x\in \alpha \cap C \neq \emptyset$. This implies that 
$0<d(x_{\alpha},x_\beta)\leq \pihalf$ (where $x_{\alpha},x_\beta$ are the centers of the respective
roots). 
If $d(x_{\alpha},x_\beta) = \pihalf$, then $x\in\alpha \cap C \subset\partial\alpha$ 
and the assertion follows because $\mu_\alpha\sin d(x,\partial\alpha)=0$.
Thus, we may assume $0<d(x_{\alpha},x_\beta)\ < \pihalf$.
This in particular excludes the case $(\mu_{\alpha},\mu_\beta)=(2,1)$ because
in this case $\Phi$ has a factor of type $BC_n$ and $\alpha,\beta\in\Phi$ are short roots,
which all have mutual distances in $\{0,\pihalf,\pi\}$.
In the remaining cases we have that $\mu_{\alpha} = \sqrt{2}\mu_\beta$ and
$\alpha$ and $\beta$ are of different type, hence,
$d(x_\beta,x_{\alpha})\in\{\frac{\pi}{4},\pihalf,\frac{3\pi}{4}\}$.
It follows that $d(x_{\alpha},x_\beta)=\frac{\pi}{4}$.
Then $\mu_{\alpha}x_{\alpha}\in\Phi$ and
$\gamma := \mu_\beta x_\beta - \mu_{\alpha} x_{\alpha}
= \mu_\beta (x_\beta - \sqrt{2} x_{\alpha})\in \Phi$
are long roots as can be seen in the root system. Also notice that $C=\gamma\cap \beta$.

Since $x\in C\subset \gamma\subset A$, we obtain $0\leq \langle x, \gamma\rangle =
\langle x, \mu_\beta x_\beta - \mu_{\alpha} x_{\alpha} \rangle$.
This in turn implies
$\mu_\beta\sin d(x,\partial\beta) = \langle x, \mu_\beta x_\beta\rangle\geq 
\langle x,\mu_{\alpha} x_{\alpha} \rangle = \mu_\alpha\sin d(x,\partial\alpha)$.
\end{proof}

The following observation will be used in the proof of Theorem~\ref{thm:topdim}.

\begin{lemma}\label{lem:inclsubcompl}
Let $K_1\supset K_2$
be two top-dimensional subcomplexes, then 
$f_{K_1}(x) \leq f_{K_2}(x)$ for all $x\in K_2$
whenever the weights defining the functions $f_{K_i}$ coincide. 
This occurs in particular if the root system is reduced.

On the other hand, if $f_{K_1}(x) > f_{K_2}(x)$ for some $x\in K_2$
and $\alpha\in\Lambda_{K_1}^{min}$ is a root such that
$f_{K_1}(x) = -\mu_\alpha\sin d(x,\partial\alpha)$, then $\alpha\in \Lambda_{K_2}^{min}$
and the weight $\mu_{i,\alpha}$ 
for the root $\alpha$ corresponding to the function $f_{K_i}$ must be 
$\mu_{i,\alpha}=i$.
\end{lemma}
\begin{proof}
 The first assertion follows directly from $\Lambda_{K_1}\subset \Lambda_{K_2}$.

For the second assertion let $\mu_{i,\beta}$ be the weights for the roots $\beta\subset S$
defining the function $f_{K_i}$.
Let $\bar f_{K_2}$ be the function in $K_2$ defined by the weights $\mu_\beta:= \mu_{2,\beta}$
if $\beta\neq \alpha$ and $\mu_\alpha := \mu_{1,\alpha}$. 
Notice that $\alpha\in \Lambda_{K_1}\subset \Lambda_{K_2}$.
Then $f_{K_1}(x) = -\mu_{\alpha}\sin d(x,\partial\alpha) \leq
\max_{\beta\in\Lambda_{K_2}}\{-\mu_\beta\sin d(x,\partial\beta)\}=\bar f_{K_2}(x)$.
Lemma~\ref{lem:weightedinrad} 
implies
$$\max_{\beta\in\Lambda_{K_2}^{min}}\{-\mu_{\beta}\sin d(x,\partial\beta)\}=
\bar f_{K_2}(x)\geq f_{K_1}(x) > 
f_{K_2}(x) =\max_{\beta\in\Lambda_{K_2}^{min}}\{-\mu_{2,\beta}\sin d(x,\partial\beta)\}.$$ 
Since $\mu_\beta = \mu_{2,\beta}$ for all roots but $\alpha$, it follows that
$\alpha\in \Lambda_{K_2}^{min}$ and $\mu_\alpha = \mu_{1,\alpha} < \mu_{2,\alpha}$.
In particular, $\mu_{i,\alpha} = i$.
\end{proof}

Let $\alpha\subset S$ be a root and let $\tau\subset\partial \alpha$ be a face.
Let $y\in \alpha$ and let $z$ be the projection of $y$ to $\partial\alpha$. 
Let $x$ be the projection of $y$ to the singular sphere $s\subset S$ spanned by $\tau$.
Then the sine rule of spherical triangles
applied to the triangle $(x,y,z)$ implies 
$\sin d(y,\partial\alpha) = \sin d(y,z) = \sin d(y,x)\sin \angle_x(y,z) =
\sin d(y,s)\sin d(\ora{sy},\partial(\Si_s\alpha))$.
We apply this observation in the following situation. 
This will be used for an induction argument in Sections~\ref{sec:subcomplexbuilding}
and \ref{sec:classicalbuild}.

Let $\tau$ be a face in the boundary of the subcomplex $K$. 
The weights $\mu_\alpha$ induce weights in the spherical Coxeter 
complex $(\Si_\tau S, Stab_W(\tau))$.
Thus, we have an induced convex function $f_{\Si_\tau K}$ in $\Si_\tau K$.

\begin{lemma}\label{lem:inductionarg}
 Let $\tau$ be a face in the boundary of the subcomplex $K$ and let 
 $s\subset S$ be the singular sphere spanned by $\tau$.
Let $y\in K$. Then $f_{CH(K,s)}(y) = \sin d(y,s)f_{\Si_s K}(\ora{sy})$.
\end{lemma}
\begin{proof}
 It follows directly from the observation above after noticing that $CH(K,s)$ is the intersection
of the roots in $\Lambda_K^{min}$ containing $s$ in their boundaries.
\end{proof}

The next lemmata describe the possible values of the function $f_K$ on vertices for the
Coxeter complexes of non-exceptional type. 
These results will be used in Sections~\ref{sec:An}, \ref{sec:Dn} and \ref{sec:classicalbuild}.

\begin{lemma}\label{lem:An}
 Let $(S,W)$ be the spherical Coxeter complex of type $A_n$.
Let $\alpha\subset S$ be a root and let $v\in \alpha $  be a vertex in its interior.
Then $\lambda_v := \sin d(v,\partial\alpha)= \cos d(v,x_\alpha)=\langle v,x_\alpha\rangle>0$ 
depends only on the type of the vertex $v$. 
In particular, if $K\subset S$ is a top-dimensional subcomplex containing
$v$ in its interior, then $f_K(v)=-\lambda_v$ is independent of $K$.
\end{lemma}
\begin{proof}
 We use the vector space realization of the Coxeter complex of type $A_n$ given in
\cite[Section 2.2.4]{LeebRamos-Cuevas:centerconj}. Then modulo the action of the Weyl group,
a vertex of type $n-k$ is $v=\frac{1}{\sqrt{k(n+1-k)(n+1)}}(n+1-k, \dots, n+1-k,-k,\dots,-k)$ 
and the center of a root is 
$x_\alpha=\frac{1}{\sqrt{2}}(e_i-e_j)$ with $i\neq j$ 
and where $(e_i)$ is the standard basis of $\R^{n+1}$.
If $0\leq d(v,x_\alpha)< \pihalf$, it readily follows that
$\lambda_v = \langle v,x_\alpha\rangle=\sqrt{\frac{n+1}{2k(n+1-k)}}$.
\end{proof}

For $n\geq 4$ we consider the Dynkin diagram of type $D_n$ with the following labelling 
\hpic{\includegraphics[scale=0.4]{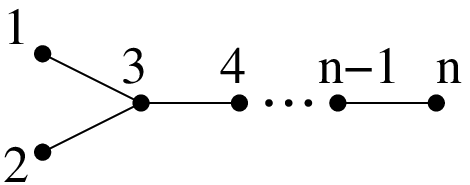}}.
Let $\lambda_i:= \frac{1}{\sqrt{2(n+1-i)}}$ for $i\neq 2$ and $\lambda_2:=\lambda_1$.

\begin{lemma}\label{lem:Dn}
 Let $(S,W)$ be the spherical Coxeter complex of type $D_n$.
Let $\alpha\subset S$ be a root and let $v\in \alpha $  be a vertex of type $i$ in its interior.
Then $\sin d(v,\partial\alpha)= \cos d(v,x_\alpha)
=\langle v,x_\alpha\rangle \in \{\lambda_i,2\lambda_i\}$. Moreover,
if $i=n$, then $\sin d(v,\partial\alpha)=\lambda_n$ and if $i\in\{1,2\}$, then
$\sin d(v,\partial\alpha)=2\lambda_1 = 2\lambda_2$.
\end{lemma}
\begin{proof}
 We use the vector space realization of the Coxeter complex of type $D_n$ given in
\cite[Section 2.2.4]{LeebRamos-Cuevas:centerconj}. 
Let $(e_i)$ be the standard basis of $\R^{n}$. Then modulo the action of the Weyl group,
a vertex of type $i$ is $v=\frac{1}{\sqrt{n+1-i}}(e_i+e_{i+1}+\dots +e_n)$ if $i\neq 2$ and
$v=\frac{1}{\sqrt{n}}(-e_1+e_{2} +e_3+\dots+ e_n)$ if $i=2$.
The center of a root is a vertex of type $n-1$ and therefore has the form 
$x_\alpha =\frac{1}{\sqrt{2}} ( \pm e_j \pm e_k )$ for $j\neq k$.
Since $\langle v, x_\alpha\rangle > 0$, the assertion follows.
\end{proof}

\begin{lemma}\label{lem:Dn2}
Let $(S,W)$ be the spherical Coxeter complex of type $D_n$.
Let $\alpha\subset S$ be a root and let $v_i,v_j\in \alpha $  be two adjacent vertices of type $i<j$.
If $j\geq 3$ and $\sin d(v_j,\partial\alpha)=2\lambda_j$, then $\sin d(v_i,\partial\alpha)=2\lambda_i$.
\end{lemma}
\begin{proof}
 With the notation in the proof of Lemma~\ref{lem:Dn}, the hypotheses imply
$v_j=\frac{1}{\sqrt{n+1-j}}(e_j+\dots +e_n)$ and
$x_\alpha =\frac{1}{\sqrt{2}} (e_k + e_l )$ with $k,l\geq j >i$.
It follows: $\sin d(v_i,\partial\alpha)=2\lambda_i$.
\end{proof}

For $n\geq 2$ we consider the Dynkin diagram of type $B_n$ with the following labelling 
\hpic{\includegraphics[scale=0.4]{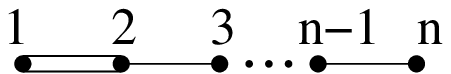}}. We have analogous results to
Lemmata~\ref{lem:Dn} and \ref{lem:Dn2}
for the Coxeter complex of type $B_n$. 
Their proofs are similar and we omit them here.

\begin{lemma}\label{lem:Bn}
Let $\Phi$ be the root system of type $B_n$, $C_n$ or $BC_n$
and let $(S,W)$ be its associated spherical Coxeter complex.
Let $\alpha\subset S$ be a root and let $v\in \alpha $  be a vertex of type $i$ in its interior.
Then $\mu_\alpha\sin d(v,\partial\alpha)=\mu_\alpha \cos d(v,x_\alpha)
=\langle v,\mu_\alpha x_\alpha\rangle \in \{\lambda_i,2\lambda_i\}$. 
Where $\lambda_i$ depends only on the type of the root system.
\qed
\end{lemma}

\begin{lemma}\label{lem:Bn2}
Let $\Phi$ be the root system of type $B_n$, $C_n$ or $BC_n$
and let $(S,W)$ be its associated spherical Coxeter complex.
Let $\alpha\subset S$ be a root and let $v_i,v_j\in \alpha $  be two adjacent vertices of type $i<j$.
If $\mu_\alpha\sin d(v_j,\partial\alpha)=2\lambda_j$, then 
$\mu_\alpha\sin d(v_i,\partial\alpha)=2\lambda_i$.
\qed
\end{lemma}

The next Lemma provides a generalization of the Lemma~\ref{lem:inclsubcompl} in the case of 
vertices in subcomplexes of positive codimension of spherical buildings of non-exceptional type.
This Lemma will be used in the proof of Theorem~\ref{thm:classicalbuild}.

\begin{lemma}\label{lem:vertexposcodim}
 Let $\Phi$ be the root system of type $A_n$, $B_n$, $C_n$, $D_n$ or $BC_n$ 
and let $(S,W)$ be its associated Coxeter complex. Let $s\subset S$ be a singular sphere.
Let $K_1,K_2\neq S$ be two top-dimensional convex subcomplexes, such that 
the interior of $K_i\cap s$ lies in the interior of $K_i$ and 
$K_2\cap s\subset K_1\cap s$.
Let $f_{K_j}$ for $j=1,2$ be the functions defined above (with possibly distinct
set of weights $\{\mu_{j,\alpha}\}$ for the roots of $S$ if $\Phi$ is non-reduced).

Let $x$ be a vertex in $K_2\cap s$. 
If $f_{K_2}(x) < f_{K_1}(x)$ and $\alpha\in\Lambda_{K_1}^{min}$ is a root 
such that $f_{K_1}(x) = -\mu_{1,\alpha}\sin d(x,\partial\alpha)$, 
then $\alpha\cap s \in \Lambda_{K_1\cap s}^{min}\cap \Lambda_{K_2\cap s}^{min}$.
\end{lemma}

\begin{proof}
First notice that for $\Phi$ of type $A_n$, 
the inequality $f_{K_2}(x) < f_{K_1}(x)$ cannot hold by Lemma~\ref{lem:An}.
Hence we may assume that $\Phi$ is not of type $A_n$.

Let $i$ be the type of the vertex $x$.
The inequality $f_{K_2}(x) < f_{K_1}(x)$ implies $f_{K_j}(x) = -j\lambda_i$, where
$\lambda_i$ takes the respective value depending on the type of $\Phi$.
We use again the usual identification of $S$ with the unit round sphere in $\R^n$.
Then modulo the action of the Weyl group and normalization of the vector
(for simplicity we will always omit the normalization factors of the corresponding vectors), we have
$x=e_i+\dots + e_n$.
Let $x_\alpha$ be the center of the root $\alpha$. 
Then $f_{K_1}(x) =  -\mu_{1,\alpha}\cos d(x,x_\alpha) = -\lambda_i$ 
implies that $x_\alpha = \varepsilon e_{\alpha_1} + e_{\alpha_2}$
with $\alpha_1 < i \leq \alpha_2$ and $\varepsilon = 0,\pm 1$ ($\varepsilon = 0$ occurs only in 
the cases $B_n,BC_n$).
After multiplying with an element of the Weyl group, we may assume
that $x_\alpha = \varepsilon e_{i-1} + e_{i}$ and $\varepsilon = 0, 1$.

Consider the sets of roots $\Pi= \{ \gamma\subset S \;|\; \mu_{2,\gamma}\cos d(x,x_\gamma) 
= 2\lambda_i\}$ and $\Upsilon = \{ \gamma\subset S \;|\;  s\subset \gamma\}$.
Then $s = \bigcap_{\gamma\in\Upsilon}\gamma$.
Let $M:=\bigcap_{\gamma\in\Pi}\gamma$.
Since $f_{K_2}(x) = -2\lambda_i$, we obtain $M\subset K_2$,
which in turn implies the inclusions 
$s\cap M \subset s\cap K_2 \subset s\cap K_1 \subset \alpha$.
Observe that if $\gamma\in\Pi$, then $x_\gamma$ is of the form $e_j$ or $e_j+e_k$ for $i\leq j<k$
and if $\gamma\in\Upsilon$, then $\langle x, x_\gamma\rangle =0$ and therefore
$x_\gamma$ is of the form $e_j$ or $\pm(e_j\pm e_k)$ for $j<k<i$ or $\pm(e_j-e_k)$ for $i\leq j<k$.

Suppose first that $\varepsilon =1$, that is, $x_\alpha = e_{i-1}+e_i$. 
Consider a point in $S$ with 
$(i-1)$-coordinate negative and let $y\neq 0$ be its projection to the subspace 
$\{z\in\R^n\;|\;\langle z,e_j\rangle = 0, j\geq i\}$. 
Then $y\notin \alpha$ and $y\in M$, therefore
$y\notin s$. 
It follows that for any point in $s$, its $(i-1)$-coordinate must be zero.
Hence $\alpha\cap s = \{\langle e_i,\cdot\rangle \geq 0\}\cap s$.

Let now $y=-e_i+e_{i+1}+\dots+e_n$, then $y\notin \alpha$.
It follows that $y\notin s$ or $y\notin M$.
In the former case, there is a root in $\Upsilon$ with center $e_i-e_j$ for some $j>i$, then
for the root $\beta\in \Pi$ centered at $e_i+e_j$ holds
$\beta\cap s = \{\langle e_i+e_j,\cdot\rangle \geq 0\}\cap s
= \{\langle e_i+e_i,\cdot\rangle \geq 0\}\cap s= \{\langle e_i,\cdot\rangle \geq 0\}\cap s
=\alpha\cap s$. 
In the latter case, there must be a root in $\Pi$ not containing $y$, the only possibility is 
the root $\beta$ centered at $e_i$. In this case
we also get $\beta\cap s = \{\langle e_i,\cdot\rangle \geq 0\}\cap s =\alpha\cap s$.

We have found a root $\beta \in \Pi$ such that 
$\beta\cap s = \alpha\cap s$. 
Let now $y$ be a point with $j$-coordinate equal 0 if $j\leq i$, or 
if $e_i-e_j\in\Upsilon$ and $j>i$, and with all other coordinates equal 1.
Suppose $y\neq 0$, that is, it defines a point in $S$. 
It follows that $y\in s\cap M$ and $y\in\{\langle e_i,\cdot\rangle = 0\}\cap s = \partial(\beta\cap s)$.
If there is another root $\gamma\in \Pi$ such that 
$y\in \partial(\gamma\cap s)$, then it must hold $\partial(\gamma\cap s) 
= \{\langle e_i,\cdot\rangle = 0\}\cap s = \partial(\beta\cap s)$. 
Therefore $\beta\cap s \in \Lambda_{M\cap s}^{min}$.
If $y=0$, then $\partial(\gamma\cap s) 
= \{\langle e_i,\cdot\rangle = 0\}\cap s = \partial(\beta\cap s)$ holds for any root $\gamma\in\Pi$.
We can again conclude that $\beta\cap s \in \Lambda_{M\cap s}^{min}$.

Since $\alpha\cap s =\beta\cap s \in \Lambda_{M\cap s}^{min} $ and we have the inclusions
$M\cap s \subset K_2\cap s\subset K_1\cap s\subset \alpha\cap s$ it follows that 
$\alpha\cap s =\beta\cap s \in\Lambda_{K_1\cap s}^{min}\cap \Lambda_{K_2\cap s}^{min}$.
\end{proof}

\subsection{Spherical buildings}

A spherical building $B$ modelled on a spherical
Coxeter complex $(S,W)$ is a $\CAT(1)$ space
together with an atlas of isometric embeddings $S \hookrightarrow B$
(the images of these embeddings are called {\em apartments})
with the following properties: any two points in $B$ are contained in a common
apartment, the atlas is closed under precomposition with isometries in $W$
and the coordinate changes are restrictions of isometries in $W$. 
We consider the empty set as a spherical building.

The objects (walls, roots,... ) defined for spherical Coxeter complexes
can be defined for the building $B$ as the corresponding images in $B$. 

A spherical building has a unique decomposition as a join of 
spherical buildings and a sphere,
whose buildings factors cannot be decomposed further. We say that the building is
{\em irreducible} if it is not a sphere and this decomposition is trivial.

A building is called {\em thick} if every wall is the boundary of at least 
three different roots.
A spherical building has a canonical 
thick structure (depending only on its isometry type)
which results from restricting to a subgroup of its Weyl group
(\cite[Sec.\ 3.7]{KleinerLeeb:quasi-isom}).

We say that an isometry of a spherical building is {\em type preserving}
if it induces the identity on the model Weyl chamber with respect to its thick structure.
We denote with $Isom_0(B)$ the group of type preserving isometries. It is a normal
subgroup of the isometry group 
$Isom(B)$ and the quotient group $Isom(B)/Isom_0(B)$ 
naturally embeds as a subgroup
of the isometry group of the model Weyl chamber (in particular, it is finite
if $B$ does not split off a spherical factor).

A {\em subbuilding} is a convex subset $B'$ of a building, such that any
two points in $B'$ are contained in a convex sphere 
$s\subset B'$ of the same dimension as $B'$.
A subbuilding carries a natural structure as a spherical building
induced by its ambient building 
(cf. \cite[Proposition 2.13]{LeebRamos-Cuevas:centerconj}).

For any point $x\in B$, the link $\Si_x B$
is again a spherical building. It decomposes as the join of a  sphere of
dimension $\dim(\tau)-1$, where $\tau$ is the smallest face of $B$ containing $x$,
and a spherical building $\Si_\tau B$ (which we call the {\em link of the face} 
$\tau$).  

Let $K\subset B$ be a top-dimensional convex subcomplex (e.g.\ an apartment)
and let $\tau\subset K$ be a face.
We denote with $St_\tau(K) \subset K$ the union of all chambers in $K$ containing $\tau$.
We call $St_\tau(K)$ the {\em star of $\tau$ in $K$}.

A point $x\in C\subset B$ in a convex subset of a spherical building is said to be an 
{\em interior point} if $\Si_x C \subset \Si_x B$ is a subbuilding and a {\em boundary point} otherwise.
The set of boundary points is denoted with $\partial C$.

\subsubsection{Root groups}

Let $B$ be a spherical building and let $\alpha\subset B$ be a root. 
The {\em root group} $U_\alpha$ associated to $\alpha$ is the group
of isometries of $B$ fixing $\alpha$ pointwise and every chamber $\sigma$
such that $\sigma\cap \alpha$ is a panel not contained in the boundary wall of $\alpha$.
Notice that $U_\alpha$ consists on type preserving isometries.

The building $B$ is called {\em Moufang} if for all roots $\alpha\subset B$,
the root group $U_\alpha$ acts transitively on apartments
containing $\alpha$.  

It is a fundamental result of Tits \cite{Tits:bn-pairs} that irreducible spherical
buildings of dimension at least $2$ are Moufang. In this case, the root group
$U_\alpha$ acts {\em simply} transitively on apartments containing $\alpha$.

Let $\sigma\subset \partial\alpha$ be a face in the boundary wall of a root 
$\alpha\subset B$. The set $\Si_\sigma \alpha$ is a root of the building $\Si_\sigma B$.
Then there is a natural restriction homomorphism 
$U_\alpha\rightarrow U_{\Si_\sigma \alpha}$. 
This homomorphism implies that the links of Moufang buildings are again Moufang.
If $B$ is irreducible, 
then by the simply transitivity of the action of $U_\alpha$ on apartments 
containing $\alpha$, this homomorphism must be injective. 
If $\Si_\sigma B$ is irreducible, 
then by the simply transitivity of the action of $U_{\Si_\sigma \alpha}$
the homomorphism must be surjective. 
In particular, if both $B$ and $\Si_\sigma B$ are irreducible,
then the root groups $U_\alpha$ and $U_{\Si_\sigma \alpha}$ are canonically isomorphic.

\subsubsection{Commutator relations}

The fact that root groups of $B$ can be canonically identified with the root groups
of the links of $B$ allows us to translate computations on the root groups in 
computations on root groups of buildings of lower dimension. In particular, we can use 
the commutator relations given in \cite{Tits:rootdata} for Moufang polygons to
deduce the commutator relations of root groups of irreducible spherical buildings
of dimension $\geq 2$. These relations also follow from the classification of spherical
buildings, but this is a much stronger result.

Let $B$ be a spherical building with associated spherical Coxeter complex $(S=S^n,W)$. 
Let $\Phi$ in $\R^{n+1}$ be a root system with the same 
associated Coxeter complex $(S,W)$.
Let $\alpha \in \Phi$ be an indivisible root.
Given a chart $(S,W)\stackrel{\iota}{\hookrightarrow} B$ for an apartment 
$A=\iota(S)\subset B$, we also denote with $\alpha$ the root $\iota(\alpha)\subset B$
(cf.\ Section~\ref{sec:rootsystem}).
Conversely, if $\alpha\subset B$ is a root, then we denote again with $\alpha$ the 
corresponding indivisible root.
There should be no confusion with this abuse of notation. 

We can now explain the commutator relations for the root groups of $B$ 
(cf.\ \cite{Tits:rootdata} and \cite[Section 3]{Timmesfeld:lie-type}).

\begin{theorem}\label{thm:commutator}
 For an irreducible spherical building $B$ of dimension $n\geq 2$, there exists a
(possibly non-reduced) 
root system $\Phi$ with the same associated Coxeter complex $(S,W)$ as $B$, such that 
the following holds. If $(S,W)\hookrightarrow B$ is a chart for an apartment $A\subset B$,
and $\alpha,\beta\in \Phi$ are two roots, then 
$$[U_\alpha, U_\beta] \subset \langle U_\gamma \;|\; 
\gamma = a\alpha + b\beta\in \Phi,\, a,b\in\N\rangle.$$
If the root system $\Phi$ is non-reduced (i.e.\ $\Phi$
is of type $BC_n$ and $(S,W)$ is of type $B_n$) and
$\alpha\in \Phi$ is a non-reduced root, then part of the assertion of the Theorem is
the existence of a subgroup $U_{2\alpha}$
of the root group $U_\alpha$ such that 
$1\neq [U_\alpha, U_\alpha]\subset U_{2\alpha}\subset Z(U_\alpha)$,
where $Z(U_\alpha)$ denotes the center of $U_\alpha$.
\end{theorem}

\begin{remark}
As an application of our results we will see that we can define $U_{2\alpha}\subset U_\alpha$
geometrically as the pointwise stabilizer of the ball of radius $\pihalf$ containing $\alpha\subset B$.
(See Proposition~\ref{prop:pihalfball}.)
\end{remark}

\subsubsection{Parabolic and unipotent subgroups}

For the rest of this section, 
let $B$ be an irreducible spherical building of dimension $\geq 2$. We denote
with $G=G_B$ the group of isometries generated by the root groups of $B$.
Then $G$ acts transitively on pairs $(\sigma, A)$, where $\sigma$ is a chamber
contained in the apartment $A\subset B$.
Moreover, $G$ is normal in $Isom_0(B)$ and $Isom_0(B) = G\cdot \hat H$, where
$\hat H = Fix_{Isom(B)}(A)$ is the pointwise stabilizer in $Isom_0(B)$ of an 
apartment $A\subset B$ (see e.g.\
\cite[3.15]{Timmesfeld:lie-type}, \cite[Thm.\ 11.36]{Weiss:sphbuild}).

Let $\sigma\subset B$ be a Weyl chamber and $A\subset B$ an apartment containing $\sigma$.
The set of {\em positive roots} $\Lambda_+$ with respect to $(\sigma, A)$ is the set
of roots contained in $A$ containing $\sigma$. For a root $\alpha\subset A$, 
we denote with $-\alpha$ the other root in $A$ with the same boundary wall as $\alpha$.
Let now $\tau\subset \sigma$ be a face. 
We call the stabilizer $P_\tau:=Stab_G(\tau)$ of $\tau$ in $G$ the 
{\em parabolic subgroup} associated to $\tau$. We denote the group
$U_\tau:=\langle U_\alpha \;|\; \tau\subset \alpha,\, \tau\not\subset (-\alpha) \rangle 
\subset P_\tau$ 
the {\em unipotent subgroup} associated to $\tau$. The unipotent subgroup
is independent of the chosen apartment $A$.
Notice that $U_\sigma=\langle U_\alpha \;|\; \alpha\in\Lambda_+ \rangle $.
The unipotent subgroup $U_\sigma$ acts simply transitively on the apartments 
containing $\sigma$.
Further, let 
$L_\tau = \langle U_\alpha, U_{-\alpha}\;|\; \tau\subset \alpha\cap(-\alpha)\rangle
\subset P_\tau$. Observe that $L_\sigma$ is trivial.
Finally, 
let $\hat P_\tau:=Stab_{Isom_0(B)}(\tau)$ be the stabilizer of $\tau$ in $Isom_0(B)$.
We have the following result relating these subgroups (see \cite[3.12]{Timmesfeld:lie-type}).

\begin{proposition}\label{prop:parabgp}
 For a face $\tau$ of a Weyl chamber $\sigma$ in an apartment $A\subset B$ holds
\begin{enumerate}[(i)]
 \item $U_\tau$ is normal in $\hat P_\tau$;
 \item $P_\tau = U_\tau L_\tau H$, in particular, $P_\sigma = U_\sigma H$;
 \item $\hat P_\tau = U_\tau L_\tau \hat H$, in particular, 
$\hat P_\sigma = U_\sigma \hat H$,
\end{enumerate}
where $H=Fix_{G}(A)$ and $\hat H = Fix_{Isom(B)}(A)$.
\end{proposition}

\begin{remark}\label{rem:unipotentpart}
The product decomposition $g = ulh \in\hat P_\tau= U_\tau L_\tau \hat H $
depends on the apartment $A$. 
We can read off the factor 
$u\in U_\tau$ from the action of $g$ on $A$ as follows. Let $\breve{\tau} \subset A$ be the face
in $A$ antipodal to $\tau$.
By the definition of $U_\tau$ we see that $St_\tau(A) \subset Fix(u)$. Hence, the convex hull of
$g\breve\tau = ulh\breve\tau = u \breve\tau$ and $St_\tau(A)$ is the apartment $A'=uA$.
Then $u$ is the unique element in $U_\tau\subset U_\sigma$ mapping the apartment $A$ 
to $CH(St_\tau(A), g\breve\tau)$.
\end{remark}

Consider now a unipotent isometry $g$, that is, $g\in U_\sigma$ for some chamber 
$\sigma \subset B$. Let again $A\subset B$ be an apartment containing $\sigma$ and
let $\hat\sigma\subset A$ be the chamber in $A$ antipodal to $\sigma$.
Let $\Gamma=(\sigma_0=\sigma, \sigma_1,\dots,\sigma_d=\hat \sigma)$
be a minimal gallery between $\sigma$ and $\hat\sigma$, that is, a sequence of chambers
of minimal length such that $\sigma_i\cap\sigma_{i+1}$ is a panel. The chambers 
$\sigma_i$ must be all contained in $A$. Then there is a unique representation
$g=g_1\dots g_d$ as the product of $g_i\in U_i:=U_{\alpha_i}$, where
$\alpha_i\subset A$ is the positive root such that  the panel $\sigma_{i-1}\cap\sigma_i$
is contained in $\partial\alpha_i$ (see \cite[Prop.\ 11.11]{Weiss:sphbuild}).
We can say more about this product representation of $g$ if we consider its fixed point set,
cf.\ Proposition~\ref{prop:coordunipotent}.

Let $K\subset A$ be a proper top-dimensional convex subcomplex.
Recall the definitions of the sets of roots $\Lambda_K^{min}\subset \Lambda_K$
in Section~\ref{sec:subcomplex}.

\begin{proposition}\label{prop:coordunipotent}
 Let $g\in U_\sigma$ be a unipotent element and let 
$A\subset B$ be an apartment containing $\sigma$.
Then
$g\in U_{\Lambda_{Fix(g)\cap A}}:=\langle U_\alpha \;|\; \alpha \in \Lambda_{Fix(g)\cap A} \rangle$.
More precisely, if $g=g_1\dots g_d$ is the product representation with
respect to the minimal gallery $\Gamma$ (cf.\ above), 
then $g_j=1$ if $\alpha_j \notin \Lambda_{Fix(g)\cap A}$.
Moreover, if $\alpha_i\in\Lambda_{Fix(g)\cap A}^{min}$ then
we can read off the $i$-coordinate $g_i$ from the action of $g$ on $\Sigma_\mu B$,
where $\mu$ is any panel contained in $\partial\alpha_i\cap Fix(g)$.
In particular, the $i$-coordinate $g_i$ is independent of the chosen minimal gallery 
$\Gamma$.
\end{proposition}
\begin{proof}
 Let $k\geq 1$ be the largest number 
such that $g_k\notin U_{\Lambda_{Fix(g)\cap A}}$. In particular,
$\alpha_k\notin\Lambda_{Fix(g)\cap A}$ and $g_k\neq 1$. 
It follows that there exists a chamber $\nu\subset Fix(g)\cap A$
such that $\nu\cap\alpha_k$ is a panel. 
This implies that $g_k\dots g_d \nu = g_k\nu =: \nu'\not\subset A$. 
Let $\Psi = (\nu_0=\nu,\nu_1,\dots,\nu_r=\sigma)$ be a minimal gallery with 
$\nu\cap\nu_1 \subset \partial \alpha_k$. 
Since $g_1\dots g_{k} \sigma =\sigma$ and $g_1\dots g_{k}\nu = g\nu =\nu$
we deduce that $g_1\dots g_{k} \Psi =\Psi$.

For $l=1,\dots,k$ let $s_l$ be the length of the chain $g_l\dots g_{k}\Psi\cap\Psi$
and write $r_l:=r-s_l$.
Observe that $g_l\dots g_{k}\Psi\cap\Psi = (\nu_{r_l},\dots, \nu_r)$.
We have just seen that $r_1=0$ and since $g_k\Psi = (\nu',\nu_1,\dots\nu_r)$, 
we have $r_k = 1$.
We show now inductively that $\nu_{r_l-1}\cap \nu_{r_l} \subset \partial \alpha_t$
for some $t\geq l$ and in the way we also show that $r_{l-1}\geq r_l$.
This yields a contradiction to $r_1=0,r_k=1$. 
Let us prove the claim. As induction basis we take
$l=k$. In this case we have $r_k=1$ and 
$\nu_0\cap\nu_1 \subset \partial \alpha_k$. 
For the induction step let us consider $l-1$.
If $\alpha_{l-1}$ contains the chamber $\nu_{r_l-1}$ then it also contains 
$\nu_{r_l}$ and the isometry $g_{l-1}\in U_{l-1}$ must fix 
$g_l\dots g_{k}\nu_{r_l-1}\not\subset A$. 
In this case, it follows that $r_{l-1}=r_l$
and $\nu_{r_{l-1}-1}\cap \nu_{r_{l-1}} = \nu_{r_l-1}\cap \nu_{r_l} 
\subset \partial \alpha_t$ with $t\geq l > l-1$ by induction.
If $\alpha_{l-1}$ does not contain the chamber $\nu_{r_l-1}$, then
there is a $m$ with $r_l-1\leq m<r$ such that $\nu_m\not\subset \alpha_{l-1}$
and $\nu_{m+1}\subset\alpha_{l-1}$, in particular,
$\nu_m\cap\nu_{m+1}\subset \partial\alpha_{l-1}$.
By induction $\nu_{r_l-1}\cap \nu_{r_l}\subset \partial \alpha_t$ with $t\geq l$, thus,
$r_l\leq m$.
It follows that $g_{l-1}g_l\dots g_{k}\nu_m = g_{l-1}\nu_m \neq \nu_m$ and therefore
$g_{l-1}g_l\dots g_{k}\Psi\cap\Psi= (\nu_{m+1},\dots, \nu_r)$.
In this case, it follows that $r_{l-1}=m+1>r_l$ and
$\nu_{r_{l-1}-1}\cap \nu_{r_{l-1}} = \nu_{m+1}\cap \nu_{m} 
\subset \partial \alpha_{l-1}$. 
This proves the claim and the first assertion of the 
proposition, that is, 
$g_k\in U_{\Lambda_{Fix(g)\cap A}}$ for all $k=1,\dots,d$ and $g\in U_{\Lambda_{Fix(g)\cap A}}$.

For the second assertion, let $\omega$ be the chamber in $A$ such that
$\omega\cap (Fix(g)\cap A) = \mu$. Then $\alpha_i$ is the only root
in $\Lambda_{Fix(g)\cap A}$ that does not contain $\omega$. 
The first part of the proposition implies that if $j\neq i$, then
$g_j$ fixes every chamber in $B$ having $\mu$ as a face.
Hence, $g\omega = g_1\dots g_i\dots g_d \omega = 
(g_1\dots g_{i-1})g_i \omega = g_i\omega$.
Therefore $g_i\in U_{\alpha_i}$ is the unique element of the root group $U_{\alpha_i}$
sending the apartment $A$ to the unique apartment containing $\alpha_i \cup g\omega$.
\end{proof}

\begin{corollary}\label{cor:fixunipotent}
 Let $g\in U_\sigma$ be a unipotent isometry. Then $g\in U_{\sigma'}$ for all chambers
$\sigma'\in Fix(g)$.
\end{corollary}
\begin{proof}
 Take an apartment $A$ containing $\sigma$ and $\sigma'$. 
Then $g\in U_{\Lambda_{Fix(g)\cap A}}\subset U_\sigma\cap U_{\sigma'}$ 
by Proposition~\ref{prop:coordunipotent}.
\end{proof}

\begin{corollary}\label{cor:star}
 Let $g$ be a unipotent isometry. Then whenever
$\tau\subset Fix(g)$ is a panel not contained in the boundary of $Fix(g)$, it holds
$St_\tau B \subset Fix(g)$.
\end{corollary}
\begin{proof}
The desired property follows from Proposition~\ref{prop:coordunipotent}
because $g$ is a product of root elements for roots $\alpha$ such that $\tau\not\subset \partial\alpha$
and each root element fixes $St_\tau B$ by definition.
\end{proof}

\section{Reducing to the irreducible case}

Let $B=B_1\circ \dots \circ B_n$ be the decomposition of the spherical building $B$
as a join of its irreducible components. 
Notice that some of the factors of $B$ may be isometric
and can be permuted by an isometry of $B$.

Let $g\in Isom(B)$ be an isometry and let $k\geq 1$ the smallest integer
with $g^k(B_1)=B_1$. Then $g$ induces an isometry of 
$B' = B_1\circ g(B_1)\circ\dots\circ g^{k-1}(B_1)\subset B$.
Let $x\in B'$ be a fixed point of $g$ and let $(a,x_0,\dots,x_{k-1})$ be its 
representation as element of the spherical join. Then $x=gx=\dots = g^{k-1}x$ implies
that $a_j=\frac{1}{\sqrt{k}}$ for $j=1,\dots,k$,
 $g^ix_0 = x_i$ for $i=0,\dots,k-1$ and $g^k x_0 = x_0$.
In particular, $x_0\in B_1$ is a fixed point of $g^k$.
It follows from Lemma~\ref{lemma:join} that the fixed point set of $g$ in $B'$
is isometric to the fixed point set of $g^k$ in $B_1$. This shows the following proposition,
which allows us to restrict our attention to irreducible spherical buildings.

\begin{lemma}
 Let $g$ be an isometry of a spherical building $B$. Then the fixed point set 
$Fix(g)\subset B$ decomposes as a spherical join, whose factors are isometric to fixed
point sets of isometries of irreducible spherical buildings. \qed
\end{lemma}

In particular, if $Fix(g)$ is not a subbuilding, then at least 
one of the factors given by the 
Proposition is not a subbuilding either. 
On the other hand, if one of these factors has circumradius 
$\leq \pihalf$, then the same is true for $Fix(g)$.

\section{Convex subcomplexes of spherical buildings}\label{sec:subcomplexbuilding}

Let $C\subset B$ be a convex subcomplex of a spherical building.
We say that an apartment $A \subset B$ {\em supports} $C$ 
if $\dim(C\cap A)=\dim C$ and $\partial(C\cap A) = \partial C\cap A$.

\begin{lemma}\label{lem:supportingapt}
 Let $C\subset B$ be a convex subcomplex and let 
$C'\subset C$ be a spherical convex subset. Then there is an apartment $A$
supporting $C$ such that $C'\subset A$. In particular, any two points in $C$ are
contained in an apartment supporting $C$.
\end{lemma}

\begin{proof}
Let $D\subset C$ be a maximal (under inclusion) spherical convex subcomplex containing $C'$.
We claim that any apartment $A$ containing $D$ supports $C$: First observe
that $C\cap A = D$ by maximality.
Since $D$ is spherical, there is a singular sphere $s$ of dimension $k:=\dim D$ containing $D$.
If $k < \dim C=:m$,  there exists
a $m$-dimensional face $\sigma\subset C$ such that $D\cap \sigma = s\cap \sigma$
has dimension $k$. The subset $s\cup\sigma\subset C$ is contained in an apartment $A'$ and
$D \subsetneq A'\cap C$, contradicting the maximality of $D$. Thus $k=m$.
If $\partial D \not\subset \partial C$, then there is a singular hemisphere $h\subset s$
containing $D$ and a $m$-dimensional face $\sigma\subset C$ such that
$h\cap\sigma = D\cap \sigma$ has dimension $m-1$.
There is an apartment $A'$ containing the subset $h\cup\sigma$. 
and $D \subsetneq A'\cap C$, contradicting again maximality.
\end{proof}

\subsection{Buildings of type $A_n$}\label{sec:An}

In the case of buildings of type $A_n$ we are able to prove that any subcomplex 
(not just a fixed point set) which is not
a subbuilding has circumradius $\leq \pihalf$.

\begin{theorem}\label{thm:An}
 Let $B$ be a (not necessarily thick) spherical building of type $A_n$. 
Let $C\subset B$ be a convex subcomplex.
Then either $C$ is a subbuilding or it has circumradius $\leq \pihalf$. 
\end{theorem}
\begin{proof}
Let $x\in C$ and let $\tau$ be the face containing $x$ in its interior. 
Hence, $\tau\subset C$ because $C$ is a subcomplex.
Let $V=\{v_1\dots,v_k\}$ be the set of vertices of $\tau$, 
then after identifying $\tau$ with a subset of the unit
round sphere in $\R^k$, we can write $x=\sum_{i=1}^k a_i v_i$ with $a_i\geq 0$.

Let $A$ be an apartment supporting $C$ containing $x$. Let $L:=C\cap A$
and let $K\subset A$ be the smallest top-dimensional 
subcomplex such that the interior of $L$ is contained
in the interior of $K$. 
Then a point $y\in L$ is in the boundary of $L$ if and only if it is in the boundary of $K$
and since $A$ supports $C$, this is also equivalent to $y$ being in the boundary of $C$.

Consider the function $f_K(x)$ defined in Section~\ref{sec:weightedincenter}.
If $\alpha\in\Lambda_K^{min}$, then 
$\sin d(x,\partial\alpha) = \cos d(x,x_\alpha) = \langle x,x_\alpha \rangle = 
\sum_{i=1}^k a_i  \langle v_i,x_\alpha \rangle = \sum_{v_i\notin \partial\alpha}a_i\lambda_i$,
where $\lambda_i:=\lambda_{v_i}$ is the constant given by Lemma~\ref{lem:An}.
It follows that
$f_K(x) =\max\limits_{\alpha\in\Lambda_K^{min}} \{-\sin d(x,\partial\alpha)\}= 
\max\limits_{\alpha\in\Lambda_K^{min}} \{-\!\!\sum\limits_{v_i\notin \partial\alpha}\!a_i\lambda_i\}= 
\max \{-\!\!\!\sum\limits_{v_i\in V-F}\! a_i\lambda_i\}$, 
where the last maximum is taken over all maximal
subsets $F\subset V$ such that the face spanned by the vertices in $F$ is contained in the boundary
of $K$, or equivalently, contained in the boundary of $C$. This implies that the function
$f(x):=f_K(x)$ is independent of the apartment $A$ supporting $C$. 

The function $f$ is nicely convex in $C$ because for any two points 
$x,y\in C$ there is an apartment $A$ supporting $C$
and containing both of them by Lemma~\ref{lem:supportingapt}, the function $f|_{C\cap A}=f_{C\cap A}$
is nicely convex in $A\cap C$ by Lemma~\ref{lem:distbound}.
It follows by Lemma~\ref{lem:radiusminimum} that $f$ has a unique minimum $x_0\in C$
and $\rad C\leq \rad(C,x_0)\leq \pihalf$.
\end{proof}

\begin{corollary}
 Let $B$ be a spherical building of type $A_n$ and $g$ a type-preserving isometry.
Then either $Fix(g)$ is a subbuilding or it has circumradius $\leq \pihalf$.
\end{corollary}

\subsection{Buildings of type $D_n$}\label{sec:Dn}

In this section we will show that a top-dimensional subcomplex of  a building of type $D_n$
is either a subbuilding or has circumradius $\leq \pihalf$. More precisely, we prove that
such a subcomplex, if it is not a subbuilding, then it has a unique incenter. 
Unfortunately, the argument used in Theorem~\ref{thm:An} for subcomplexes of positive codimension
of buildings of type $A_n$
does not work in this case as we will illustrate in Example~\ref{ex:Dn}.

\begin{lemma}\label{lem:vertDn}
 Let $B$ be a (not necessarily thick) spherical building of type $D_n$. 
Let $C\subset B$ be a top-dimensional convex subcomplex which is not a subbuilding.
Let $x\in C$ be a vertex and let $A\subset B$ be an apartment supporting $C$ and $x\in A$. 
Then the function $f(x) := f_{C\cap A} (x)$ as defined 
in Section~\ref{sec:weightedincenter} does not depend on the choice of the apartment $A$.
\end{lemma}
\begin{proof}
If $x\in \partial C$, then $f(x)=0$ and the assertion follows. So we assume that $x$ is in
the interior of $C$.
Let $A'\subset B$ be another apartment containing $x$ and supporting $C$. 
By Lemma~\ref{lem:supportingapt}, we can find a sequence $A=A_0,\dots, A_m=A'$ of
apartments supporting $C$ and containing $x$ such that $A_i\cap A_{i+1}$ is a root.
Thus, we may assume that $\alpha:=A\cap A'$ is a root.

Suppose that $f_{C\cap A'} (x) > f_{C\cap A}(x)$. Let $i$ be the type of the
vertex $x$. 
Then by Lemma~\ref{lem:Dn}, $f_{C\cap A'} (x) = -\lambda_i > -2\lambda_i = f_{C\cap A}(x)$
and $3\leq i \leq n-1$.
We identity the apartments $A,A'$ simultaneously with the unit sphere $S^{n-1}\subset \R^n$
such that the centers of roots correspond to points
$\frac{1}{\sqrt{2}} ( \pm e_j \pm e_k )$,
 the identifications coincide
in $\alpha=A\cap A'$ and $x$ corresponds to the point 
$\frac{1}{\sqrt{n+1-i}}(e_i+e_{i+1}+\dots +e_n)$. 
To simplify the notation, we omit the normalizing factor, so we write
$x = e_i+e_{i+1}+\dots +e_n$.
Observe that $x\in\alpha$ implies that the center $x_\alpha$ of $\alpha$ 
is of the form $e_j\pm e_k$ for $i\leq j,k$ and $j\neq k$;
$\pm e_j + e_k$ for $j<i<k$; or
$\pm e_j \pm e_k$ for $j<k<i$.
Since $f_{C\cap A'} (x) = -\lambda_i $, there must be a root $\beta\in\Lambda_{C\cap A'}$
centered at $x_\beta$ such that $\cos d(x,x_\beta) = \lambda_i $.
It follows that $x_\beta$ has the form $\pm e_r + e_s$ with $1\leq r <i$ and $i\leq s \leq n$.
Suppose $x_\beta=e_r + e_s$, the other case is similar.
On the other hand, $f_{C\cap A} (x) = -2\lambda_i $ implies that
$\bigcap_{\gamma\in\Pi_i} \gamma \subset C\cap A $ for 
$ \Pi_i:=\{\langle e_j+e_k, \cdot \rangle \geq 0 \;|\; i\leq j< k\}$. 
It follows that 
$\alpha \cap (\bigcap_{\gamma\in\Pi_i} \gamma) \subset C\cap A \cap A'  \subset C\cap A' \subset\beta$.
Notice that since $A$ and $A'$ both support the subcomplex $C$, 
the roots $\alpha$ and $\beta$ must be different.
This implies (see the different possibilities for $x_\alpha$ above) that there exists a point 
$y\in\alpha\cap \{ (v_1,\dots,v_n)\in S^{n-1} \;|\; v_r=-2, v_i=v_{i+1}=\dots=v_n=1\} \neq \emptyset$. 
Then $y\notin \beta$ and $y\in \alpha\cap(\bigcap_{\gamma\in\Pi_i} \gamma)$. 
We get a contradiction. Hence, $f_{C\cap A'} (x) \leq f_{C\cap A}(x)$. 
Interchanging the roles of $A$ and $A'$ we obtain the equality $f_{C\cap A'} (x) = f_{C\cap A}(x)$.
\end{proof}

\begin{theorem}\label{thm:Dn}
Let $B$ be a (not necessarily thick) spherical building of type $D_n$. 
Let $C\subset B$ be a top-dimensional convex subcomplex which is not a subbuilding.
Then $-\sin d(\cdot,\partial C)$ is a nicely convex function on $K$ and $K$ has a 
unique incenter $x_0$. In particular, $C$ has circumradius $\leq \pihalf$.
\end{theorem}
\begin{proof}
Let $x\in C$ and let $\tau\subset C$ be the face containing $x$ in its interior.
Let $v_{i_1},\dots, v_{i_k}$ be the vertices of $\tau$ with $v_{i_j}$ of type $i_j$
and $i_1<\dots<i_k$. After identifying $\tau$ with a subset of the unit
round sphere in $\R^k$, we can write $x=\sum_{j=1}^k a_j v_{i_j}$ with $a_j\geq 0$.

Let $A\subset B$ be an apartment supporting $C$ and $x\in A$. 
We want to prove that the function $f(x) := f_{C\cap A} (x)$ as defined 
in Section~\ref{sec:weightedincenter} does not depend on the choice of the apartment $A$. 
The conclusion of the theorem then follows from Lemmata~\ref{lem:supportingapt}
and \ref{lem:distbound}.
So, let $A'$ be another such apartment and suppose $f_{C\cap A} (x) < f_{C\cap A'} (x)$.

Let $\alpha\in \Lambda_{C\cap A'}$ be a root such that $f_{C\cap A'} (x) = -\sin d(x,\partial\alpha)$.
Suppose first that there is a vertex $v=v_{i_j}$ of $\tau$ with $v\in \partial\alpha$.
Then by Lemma~\ref{lem:inductionarg}, we have
$f_{C\cap A'}=\sin d(x,v)f_{\Si_v(C\cap A')}(\ora{vx})$.
Since the links of a building of type $D_n$ are spherical joins of buildings of type $D$ and $A$,
using induction on $n$ and Theorem~\ref{thm:An}, we obtain
$f_{\Si_v(C\cap A')}(\ora{vx}) = f_{\Si_v(C\cap A)}(\ora{vx})$.
It follows again by Lemma~\ref{lem:inductionarg} that
$f_{C\cap A}(x) \geq \sin d(x,v)f_{\Si_v(C\cap A)}(\ora{vx}) =
\sin d(x,v)f_{\Si_v(C\cap A')}(\ora{vx}) = f_{C\cap A'}(x)$, 
contradicting our first assumption $f_{C\cap A} (x) < f_{C\cap A'} (x)$.
Hence, $\tau$ is contained in the interior of $\alpha$.

Recall that by Lemma~\ref{lem:Dn}, $\sin d(v_{i_j},\partial\alpha)$ can only take at most the 
two values $\lambda_{i_j}, 2\lambda_{i_j}$.
Let $r$ be the smallest number such that $\sin d(v_{i_r},\partial\alpha)=\lambda_{i_r}$.
In particular, $i_r\geq 3$ by Lemma~\ref{lem:Dn}.
Then by Lemma~\ref{lem:Dn2}, $\sin d(v_{i_j},\partial\alpha)=\lambda_{i_j}$ for all $j\geq r$.
Therefore
\begin{align*}
 f_{C\cap A'}(x) &= -\sin d(x,\partial\alpha)= -\sum_{j=1}^k a_j\sin d(v_{i_j},\partial\alpha) 
=-2\sum_{j=1}^{r-1}a_j\lambda_{i_j}  -\sum_{i=r}^ka_j\lambda_{i_j} \\
& > f_{C\cap A}(x) = \max_{\beta\in\Lambda_{C\cap A}}\{-\sin d(x,\partial\beta)\} =
\max_{\beta\in\Lambda_{C\cap A}}\{-\sum_{j=1}^k a_j\sin d(v_{i_j},\partial\beta)  \}.
\end{align*}
It follows that for each $\beta\in\Lambda_{C\cap A}$, there must be a $j\geq r$ such that
$\sin d(v_{i_j},\partial\beta) = 2\lambda_{i_j}$.
Then again by Lemma~\ref{lem:Dn2}, $\sin d(v_{i_r},\partial\beta) = 2\lambda_{i_r}$
for all $\beta\in\Lambda_{C\cap A}$. 
Thus, $f_{C\cap A}(v_{i_r}) = -2\lambda_{i_r} < -\lambda_{i_r} = 
-\sin d(v_{i_j},\partial\alpha) \leq f_{C\cap A'}(v_{i_r})$, contradicting Lemma~\ref{lem:vertDn}.
\end{proof}

\begin{example}\label{ex:Dn}
 Consider a building of type $D_4$ and the convex subcomplex $C$ consisting of a segment 
$c_1$ with vertices of type $31313$ and a segment $c_2$
with vertices $131$, which intersect in their midpoints. Let $x$ be their common midpoint. 
Let $A_i$ be apartments containing $c_i$. Then $A_i$ supports $C$.
Let $K_i\subset A_i$ be the smallest top-dimensional subcomplex such
 that the interior of $c_i$ is contained in the interior of $K_i$. 
Then $K_1$ is a root and $x$ is its center. It follows that $f_{K_1}(x) = -1 < f_{K_2}(x)$.
\end{example}

\section{Fixed point sets of unipotent isometries}
\label{sec:unipotentisom}

Let $B$ be an irreducible spherical building of dimension at least $2$ and
let $g\neq 1$ be a unipotent isometry. Let $A\subset B$ be an apartment
such that $Fix(g)\cap A$ is a top-dimensional subset. 
Then by Corollary~\ref{cor:incenter}
$Fix(g)\cap A$ has a unique incenter. If $B$ is of simply-laced type we will
prove that $Fix(g)$ also has a unique incenter. However, this is no longer true for
other types of buildings. Nevertheless, we will show that $Fix(g)$ has always
a unique {\em weighted} incenter in the sense of Section~\ref{sec:weightedincenter}.

Let $g$ be an unipotent isometry of $B$.
Let $\Phi$ the root system associated to $B$ by Theorem~\ref{thm:commutator}.
We consider now the top-dimensional convex subcomplex 
$K:=Fix(g)\cap A\subset A$ for some apartment $A\subset B$.
We want to define the weighted incenter of $K$ as in Section~\ref{sec:weightedincenter}.
For this, we have to define the corresponding weights for non-reduced roots.
Let $\alpha\subset A$  be a root such that the corresponding indivisible root
$\alpha\in \Phi$ is non-reduced.
We define the weight $\mu_\alpha$ in dependency on $g$ as follows.
First, if $\alpha\notin\Lambda_K$, we set $\mu_\alpha:=2=\Vert{2\alpha}\Vert$; and
if $\alpha\in\Lambda_K - \Lambda_K^{min}$, we set $\mu_\alpha:=1=\Vert{\alpha}\Vert$.
Finally,
if $\alpha\in\Lambda_K^{min}$, we set the weight $\mu_\alpha:=2$, if
the $\alpha$-coordinate $g_\alpha$ of $g$ (see Proposition~\ref{prop:coordunipotent}) 
is in $U_{2\alpha}$ (see Theorem~\ref{thm:commutator}); and we set 
$\mu_\alpha:=1$, if $g_\alpha \in U_\alpha - U_{2\alpha}$.
Notice that by Proposition~\ref{prop:coordunipotent}, the weights $\mu_\alpha$ do not
depend on the apartment $A$ containing $\alpha$.

\begin{theorem}\label{thm:unipotentisom}
Let $B$ be an irreducible spherical building of dimension at least $2$ and
let $g\neq 1$ be a unipotent isometry. Let $x\in Fix(g)$ and	 let $A\subset B$ be an apartment
containing $x$. Then the function $f(x) := f_{Fix(g)\cap A} (x)$ as defined 
in Section~\ref{sec:weightedincenter} with the weights given above does not
depend on the choice of the apartment $A$. In particular, $f$ defines a nicely convex function
in $Fix(g)$ and it has a unique minimum $x_0\in Fix(g)$, the {\em weighted incenter} of 
$Fix(g)$. Moreover, $\rad(Fix(g),x_0)\leq \pihalf$.
\end{theorem}
\begin{proof}
 Let $A'\subset B$ be another apartment containing $x$. We may assume that
there is a chamber $\sigma\subset Fix(g)$ with $x\in\sigma\subset A\cap A'$. 
Then there is a unipotent element $u\in U_\sigma$ such that $u A = A'$.
Let $\Lambda_K$ be the set of the positive roots in $A$ containing $K:=Fix(g)\cap A$ and let
$\Lambda_{K'}$ be the set of the positive roots in $A'$ containing $K':=Fix(g)\cap A'$.

Let $\Pi'$ be a set of positive roots in $A'$ such that 
$g\in U_{\Pi'}=\langle U_\alpha\;|\;\alpha\in\Pi'\rangle$.
Then $M=\bigcap_{\alpha\in\Pi'}\alpha\subset Fix(g)\cap A' = K'$ and by 
Lemma~\ref{lem:weightedinrad},
\begin{align*}
 f_{M}(x) = \max_{\alpha\in \Lambda_M^{min}}\{-\mu_\alpha\sin d(x,\partial\alpha)\} =
\max_{\alpha\in \Pi'}\{-\mu_\alpha\sin d(x,\partial\alpha)\} &\\
 =\max_{\alpha\in \Lambda_M}\{-\mu_\alpha\sin d(x,\partial\alpha)\} &\geq f_{K'}(x).
\end{align*}

 Our goal is to find a $\Pi'$ as above such that $f_K(x)\geq -\mu_\alpha\sin d(x,\partial\alpha)$ for
all $\alpha\in\Lambda_M^{min}\subset \Pi'$. 
From this, it follows that $f_K(x)\geq  f_M(x)\geq f_{K'}(x)$. Switching the roles of $A,A'$
we also deduce $f_{K'}(x)\geq  f_{K}(x)$ and therefore we obtain the equality $f_K(x) = f_{K'}(x)$.

Without loss of generality we may assume that $u=:g_0\in U_{\alpha_0}$ for some positive root 
$\alpha_0\subset A$. Choose some minimal gallery from $\sigma$ to its antipodal chamber in $A$ and
let $g=g_1\dots g_d$ with $g_i\in U_{\alpha_i}$ 
be the product representation of $g$ with respect to this gallery.
Then Proposition~\ref{prop:coordunipotent} implies that for $i=1,\dots,d$,
if $\alpha_i\notin \Lambda_K$ then $g_i=1$.
For $i=0,1,\dots,d$
write $\beta_i=2\alpha_i$ if $\alpha_i$ is non-reduced and $g_i\in U_{2\alpha_i}$
(see Theorem~\ref{thm:commutator}) and $\beta_i=\alpha_i$ otherwise.
By Theorem~\ref{thm:commutator} we have
$u^{-1}g_iu = g_ih_i$ with $h_i\in\langle U_\gamma \;|\; 
\gamma = a\beta_0 + b\beta_i \in \Phi,\, a,b\in\N\rangle$.
Let $\Pi$ be the set of roots $\delta\subset A$ such that for the corresponding indivisible
root $\delta\in \Phi$ holds that $\delta$ or $2\delta$ is in
$\{ \gamma\in \Phi \;|\;\gamma = a\beta_0 + b\beta_i \in \Phi;\; 
\alpha_i\in\Lambda_K;\; a\in \Z_{\geq 0},\; b\in\N \}$.
It follows that $u^{-1}gu\in \langle U_\gamma \;|\; \gamma \in \Pi\rangle$.
Notice that $\Lambda_K\subset \Pi$.

Let $\Pi'=\{u\gamma\;|\; \gamma\in\Pi\}$. Then $\Pi'$ is a set  of roots in $A'$
and since $U_{u\gamma}=uU_\gamma u^{-1}$, we obtain
$g\in U_{\Pi' }=\langle U_\alpha\;|\;\alpha\in\Pi'\rangle$.
We now verify that $\Pi'$ has the desired properties, that is,
$f_K(x)\geq -\mu_\alpha\sin d(x,\partial\alpha)$ for
all $\alpha\in\Lambda_M^{min}\subset \Pi'$. 

We give first the argument for $\Phi$
reduced because it is much simpler, although we could just omit it, since the argument in the 
non-reduced case works in general. So suppose $\Phi$ is reduced.
Take $u\gamma\in \Pi'$. Let $x_\gamma,x_{\alpha_i}$ be the centers of the respective
roots. In the reduced case we have $\beta_i=\alpha_i$ and 
$\gamma = a\alpha_0 + b\alpha_j\in\Phi$ for some $\alpha_j\in\Lambda_K$.
Further, $\mu_{u\gamma}=\mu_\gamma$. Identify as usual the apartment $A$ with the unit sphere.
Then $\mu_{u\gamma}x_\gamma = \gamma, \mu_{\alpha_j}x_{\alpha_j} = \alpha_j\in \Phi$.
It follows that $\langle x,\mu_{u\gamma}x_\gamma \rangle -\langle x, \mu_{\alpha_j}x_{\alpha_j}\rangle
= \langle x , \gamma -\alpha_j \rangle = \langle x, a\alpha_0 + (b-1)\alpha_j\rangle \geq 0$
because $\alpha_0,\alpha_j$ are positive roots and $a\geq 0,b\geq 1$. 
This implies that
$f_K(x)\geq -\mu_{\alpha_j}\sin d(x,\partial\alpha_j) = -\langle x, \mu_{\alpha_j}x_{\alpha_j}\rangle
\geq -\langle x,\mu_{u\gamma}x_\gamma \rangle 
= -\mu_{u\gamma}\sin d(x,\partial\gamma)=-\mu_{u\gamma}\sin d(x,\partial(u\gamma))$.

We consider now the general case. 
Take $u\gamma\in \Lambda_M^{min}\subset \Pi'$ and let 
$x_\gamma,x_{\alpha_i}$ be the centers of the respective roots.
In this case we have several possibilities:
$\mu_{u\gamma}x_\gamma,a\beta_0 + b\beta_j\in\{\gamma,2\gamma\}\subset \Phi$
for some $\alpha_j\in\Lambda_K$
and $\mu_{\alpha_j}x_{\alpha_j},\beta_j\in \{\alpha_j,2\alpha_j\}\subset \Phi$.
It follows that $\mu_{u\gamma}x_\gamma = c(a\beta_0 + b\beta_j)$ and
$\mu_{\alpha_j}x_{\alpha_j} = c'\beta_j$ for some $c,c'\in\{\frac{1}{2},1,2\}$.
Hence,
$f_K(x) + \mu_{u\gamma}\sin d(x,\partial(u\gamma)) \geq 
-\mu_{\alpha_j}\sin d(x,\partial\alpha_j)+ \mu_{u\gamma}\sin d(x,\partial(u\gamma))=
\langle x,\mu_{u\gamma}x_\gamma \rangle -\langle x, \mu_{\alpha_j}x_{\alpha_j}\rangle
= \langle x, ca\beta_0 + (cb-c')\beta_j\rangle$.
Since $\alpha_0,\alpha_j$ are positive roots and $a\geq 0,b\geq 1$, it suffices to show
that $cb-c'\geq 0$.

Suppose first that $c'=2$. Then $\alpha_j\in\Phi$ is non-reduced and
$\mu_{\alpha_j}x_{\alpha_j}=2\beta_j=2\alpha_j$. 
In particular, $\mu_{\alpha_j}=2$.
By the definition of the weights, it follows that $\alpha_j\notin \Lambda_K$; or,
$\alpha_j\in\Lambda_K^{min}$ and $g_{\alpha_j}=g_j\in U_{2\alpha_j}$.
The former cannot happen by the definition of $\Pi$ and the latter implies 
$\beta_j=2\alpha_j$ by the definition of $\beta_j$. We get a contradiction, thus,
$c'\leq 1$.

Suppose now that $c=\frac{1}{2}$. Then $\gamma\in\Phi$ is non-reduced and
$\mu_{u\gamma}x_\gamma = \gamma = \frac{1}{2}(a\beta_0 + b\beta_j)$.
In particular, $\mu_{u\gamma}=1$. This implies that
$u\gamma\in \Lambda_{K'} - \Lambda_{K'}^{min}$; or,
$u\gamma\in \Lambda_{K'}^{min}$ and $g_{u\gamma}\in U_{u\gamma}-U_{2u\gamma}$.
The former cannot happen because $u\gamma\in \Lambda_M^{min}$ and $M\subset K'$.
Thus, $u\gamma \in \Lambda_M^{min}\cap \Lambda_{K'}^{min}$.
Let $\tau\subset A'$ be a panel in $(-u\gamma)\cap(M\cap K')$.
Then by Proposition~\ref{prop:coordunipotent}, we can read off the element $g_{u\gamma}$
from the action of $g$ on $\Sigma_\tau B$. 
Since $\tau\subset (-u\gamma)\cap M$ lies on the boundary of $M$ 
and $M=\bigcap_{\alpha\in\Pi'}\alpha$,
the only root group $U_\delta$ for $\delta\in\Pi'$ that acts non-trivially  on $\Sigma_\tau B$ is
$U_{u\gamma}$.
Recall that $g\in\langle uU_\delta u^{-1}\;|\; \delta=p\beta_0+q\beta_i\in\Phi;\; \alpha_i\in\Lambda_K;\;
p\in\Z_{\geq 0};\; q\in\N\rangle$.
Suppose that for all $k$ such that $p\beta_0+q\beta_k\in\{\gamma, 2\gamma\}$ for some 
$p\geq0,q\geq1$ follows that $p\beta_0+q\beta_k = 2\gamma$.
This would imply that $g$ is a product of elements in $U_\delta$ for $\delta\in\Pi'-\{u\gamma\}$ and
elements in $U_{2u\gamma}$. This in turn would imply that
the action of $g$ on $\Sigma_\tau B$ is given by the action of an element in $U_{2u\gamma}$ on
$\Sigma_\tau B$. This contradicts the fact $g_{u\gamma}\in U_{u\gamma}-U_{2u\gamma}$.
Hence, there is a $k$ such that $p\beta_0+q\beta_k = \gamma$. 
From this we see that after replacing $j$ with $k$ we may assume that $c\geq 1$.

Finally we can see that $cb-c'\geq 1b-1\geq 1-1 =0$ and conclude that
$f_K(x) + \mu_{u\gamma}\sin d(x,\partial(u\gamma))\geq 0$. This is what remained to be proved.

The remaining assertions of the theorem are just a consequence of Lemma~\ref{lem:radiusminimum}.
\end{proof}

\begin{remark}\label{rem:polygons}
Although we are mainly interested in buildings of dimension $\geq 2$, the results of this section
remain valid for Moufang generalized triangles and quadrangles with the proofs unchanged.
The reason is that the commutator relations are still valid in these cases (\cite{Tits:rootdata},
\cite{Timmesfeld:lie-type}).
\end{remark}

\section{Top-dimensional fixed point sets}\label{sec:topdim}

In this section let again $B$ be an irreducible spherical building of dimension at least $2$.
Let $g\in Isom(B)$ be an isometry such that $Fix(g)\subset B$ is a top-dimensional
subcomplex which is not a subbuilding. Let $A\subset B$ be an apartment such that 
$Fix(g)\cap A$ is top-dimensional. Then we can define the function
$f_{Fix(g)\cap A}$ in $Fix(g)\cap A$ as in Section~\ref{sec:unipotentisom}, but it is no longer true
that $f_{Fix(g)\cap A}(x)$ does not depend on the apartment $A$ containing $x\in Fix(g)$.
However, we can rescue the argument if we consider only some specials apartments.

If $A$ is an apartment with $Fix(g)\cap A$ top-dimensional,
we say that $A$ {\em supports} $g$ if $A$ supports $Fix(g)$ 
(cf.\ Section~\ref{sec:subcomplexbuilding})
and additionally the following holds:
if $\alpha\subset A$ is a non-reduced root in $\Lambda_{Fix(g)\cap A}^{min}$ and
$\tau\subset \partial\alpha\cap Fix(g)$ is a panel, then
if there is an apartment $A'$ containing $\alpha$
such that the unique element in $U_\alpha$ sending 
$\Si_\tau A'$ to $\Si_\tau gA'$ lies in $U_{2\alpha}$, then
the unique element in $U_\alpha$ sending $\Si_\tau A$ to $\Si_\tau gA$ also lies in $U_{2\alpha}$.
Lemma~\ref{lem:supportingapt} readily generalizes to apartments supporting $g$, since
in the notation above clearly $A'$ also supports $Fix(g)$.

Observe that if $u$ is a unipotent isometry, then any apartment, which intersects
$Fix(u)$ in a top-dimensional set, supports $u$.

Before we continue, let us explain our motivation to consider this special kind of apartments.
Let $X=G/K$ be a symmetric space of noncompact type and let $g\in G=Isom_0(X)$ be an isometry.
Then $g$ has a Jordan decomposition $g=su$ such that $s$ is semisimple, $u$ is unipotent and
they commute with each other. 
The minimal set $Min(s)$ (i.e.\ the set where the displacement function of $s$ attains its minimum)
is a totally geodesic subspace. 
The boundary at infinity $\tits Min(s)$ is a subbuilding of $\tits X$
and $u$ acts on $\tits Min(s)$ as a unipotent isometry. Further,
$Fix_\infty(g)=Fix_\infty(s)\cap Fix_\infty(u)=\tits Min(s)\cap Fix_\infty(u)$,
 where $Fix_\infty$ denotes the fixed point set in $\tits X$.
(see \cite[Proposition 4.1.5]{Eberlein:nonpositive} and the discussion in Section~\ref{sec:jordan}). 
We can then apply Theorem~\ref{thm:unipotentisom} to the action of $u$ on 
the building $\tits Min(s)$ to conclude
that $Fix_\infty(g)$ has circumradius $\leq \pihalf$.
 In general, we do not have a Jordan decomposition, so we use the union of the apartments 
supporting $g$ as a substitute of the subbuilding $\tits Min(s)$.

 We now return to our original discussion.
Let now $A$ be an apartment supporting $g$.
By Proposition~\ref{prop:parabgp} we can write $g=uh$ with $u\in U_\sigma$ 
for all $\sigma\subset Fix(g)\cap A$ and $h\in \hat H = Fix_{Isom(B)}(A)$.
In particular, $Fix(g)\cap A = Fix(u)\cap A$ and $gA = uA$.
Now if $\alpha\subset A$ is a non-reduced root, then we define the weight $\mu_\alpha$
as the corresponding one for the unipotent isometry $u$.
With these weights we obtain the functions 
$f_{Fix(g)\cap A} (x)$ as defined in Section~\ref{sec:weightedincenter}.

\begin{theorem}\label{thm:topdim}
Let $B$ be an irreducible spherical building of dimension at least $2$ and
let $g\in Isom(B)$ be an isometry such that $Fix(g)\subset B$ is a top-dimensional
subcomplex which is not a subbuilding.
Let $x\in Fix(g)$ and let $A\subset B$ be an apartment
containing $x$ and supporting $g$. 
Then the function $f(x) := f_{Fix(g)\cap A} (x)$ 
 does not depend on the choice of the apartment $A$ supporting $g$. 
In particular, $f$ defines a nicely convex function
in $Fix(g)$ and it has a unique minimum $x_0\in Fix(g)$, the {\em weighted incenter} of 
$Fix(g)$. Moreover, $\rad(Fix(g),x_0)\leq \pihalf$. 
\end{theorem}

\begin{proof}
Let $A'$ be another apartment supporting $g$ and containing $x\in Fix(g)$.
We may assume that $A\cap A'$ contains a chamber $x\in\sigma\subset Fix(g)$. 
Write $g=uh$ with $u\in U_\sigma$ and $h\in \hat H = Fix_{Isom(B)}(A)$
and $g=u'h'$ with $u'\in U_\sigma$ and $h'\in \hat H' = Fix_{Isom(B)}(A')$.

Consider $Fix(u)\cap A'$. We claim that it is contained in $Fix(h)\cap A'$:
Suppose not, then there is a positive root $\alpha\subset A'$ such that 
$\alpha\in\Lambda_{Fix(h)\cap A'}^{min}$
and $\alpha\notin\Lambda_{Fix(u)\cap A'}$.
Let $\tau$ be a panel in $\partial\alpha\cap Fix(u)\cap Fix(h)$.
Let $\nu_0, \nu_1$ be the chambers in $-\alpha\subset A'$ and $\alpha$ respectively
containing $\tau$. Then $\nu_1\subset Fix(h)\cap Fix(u)$, thus $\nu_1\subset Fix(g)$.
Further, $\alpha\notin\Lambda_{Fix(u)\cap A'}$ implies that $\nu_0\subset Fix(u)$, this in turn implies 
that $u$ fixes every chamber containing $\tau$ by Corollary~\ref{cor:star}.
Also, $\alpha\in\Lambda_{Fix(h)\cap A'}^{min}$ implies that $\nu_0\not\subset Fix(h)$.
It follows that $g\nu_0 = uh\nu_0 = h\nu_0 \neq \nu_0$ and $\nu_0\not\subset Fix(g)$.
Therefore $\tau\subset \partial(Fix(g)\cap A')\subset \partial Fix(g)$ because $A'$ supports $Fix(g)$.
On the other hand, $Fix(h)\subset B$ is a subbuilding and therefore there exists a chamber
$\nu'\subset Fix(h)$ such that $\nu'\cap\nu_1=\tau$. But $u$ must also fix $\nu'$, hence,
$g=uh$ fixes $\nu'$ as well. This contradicts the fact that $\tau\subset \partial Fix(g)$.
So we conclude that $Fix(u)\cap A' \subset Fix(h)\cap A'$ and in particular
$Fix(u)\cap A' \subset Fix(g)\cap A'$.

In the non-reduced case we have to pay special attention in how the weights are defined.
For $\beta\subset A'$, let $\bar\mu_\beta$ be the weight as defined for the unipotent isometry $u$
and let $\bar f_{Fix(g)\cap A'} (x) : = 
\max_{\beta\in \Lambda_{Fix(g)\cap A'}} \{-\bar\mu_\beta\sin d(x,\partial\beta)\}$. Then
Theorem~\ref{thm:unipotentisom} applied to $u$ implies
$$f_{Fix(g)\cap A}(x) = f_{Fix(u)\cap A}(x) = \bar f_{Fix(u)\cap A'}(x).$$
By Lemma~\ref{lem:weightedinrad} and the fact 
$Fix(u)\cap A' \subset Fix(g)\cap A' = Fix(u')\cap A$
(cf.\ Lemma~\ref{lem:inclsubcompl}), 
if we want to show that  $\bar f_{Fix(u)\cap A'}(x)\geq f_{Fix(g)\cap A'}=f_{Fix(u')\cap A'}$, 
we just have to show
that for every root $\beta\in 
\Lambda_{Fix(u')\cap A'}^{min} \cap  \Lambda_{Fix(u)\cap A'}^{min}$ holds 
$\mu_\beta \geq \bar\mu_\beta$.
This is clear unless $\beta$ is non-reduced. Let us consider this case.
Since $Fix(u)\cap A' \subset Fix(h)\cap A'$, we can take a chamber $\omega\subset Fix(h)$
such that $\omega\cap \beta$ is a panel $\pi$.
Let also $\omega'\subset A'\subset Fix(h')$ be the chamber such that $\omega'\cap \beta = \pi$.
Recall the definition of the weights:
We have $\mu_\beta=2$ if the element $(\omega'\mapsto g\omega' = u'\omega') \in U_\beta$ lies in 
$U_{2\beta}$ and $\mu_\beta=1$ otherwise;
similarly, $\bar\mu_\beta=2$ if the element $(\omega\mapsto g\omega = u\omega) \in U_\beta$ lies in 
$U_{2\beta}$ and $\bar\mu_\beta=1$ otherwise.
Therefore $\bar\mu_\beta=2$ implies $\mu_\beta=2$ by the condition about non-reduced roots
in the definition of an apartment supporting $g$ and because $A'$ supports $g$.
Hence, $\mu_\beta \geq \bar\mu_\beta$.
So we can conclude
$$f_{Fix(g)\cap A} = f_{Fix(u)\cap A} =\bar  f_{Fix(u)\cap A'} \geq f_{Fix(g)\cap A'}.$$
Exchanging the roles of $A,A'$ we obtain $f_{Fix(g)\cap A}=f_{Fix(g)\cap A'}$.
The second part of the assertion follows from Lemmata~\ref{lem:supportingapt} and 
\ref{lem:radiusminimum}.
\end{proof}

Theorem~\ref{thmintro:topdim} from the introduction follows directly from 
Theorem~\ref{thm:topdim} and Lemma~\ref{lemma:join}.

\begin{corollary}\label{cor:fixedptssubgp}
 Let $H\subset Isom(B)$ be a subgroup of isometries such that 
the fixed point set $Fix(H)$ is top-dimensional. Suppose that there is an element $g\in H$
such that $Fix(g)$ is not a subbuilding. Then $\rad (Fix(H)) \leq \pihalf$.
\end{corollary}
\begin{proof}
 By Theorem~\ref{thm:topdim}, $Fix(g)$ has a circumcenter $x_0$ such that
$\rad(Fix(g),x_0)\leq \pihalf$. Since $Fix(H)\subset Fix(g)$, 
it follows that $Fix(H) \subset B_{\pihalf}(x_0)$.
The distance from $x_0$ to points in $Fix(H)$ cannot be constant $\pihalf$ because $Fix(H)$
is top-dimensional. This implies that the projection $\bar x_0$ of $x_0$ into $Fix(H)$ is well defined
and $\rad(Fix(H))\leq \rad(Fix(H),\bar x_0)\leq \pihalf$.
\end{proof}

\section{Fixed point sets in non-exceptional buildings}\label{sec:classicalbuild}

In this section we consider fixed point sets of any codimension and show that 
fixed point sets of isometries of spherical buildings without factors of exceptional type
are either subbuildings or have circumradius $\leq \pihalf$.
The proof relies on Lemma~\ref{lem:vertexposcodim}. 
This Lemma does not hold for the Coxeter complex of type $F_4$. We have not find
counterexamples for the types $E_k$, $k=6,7,8$, but we also have no reason to believe that
they do not exist. Hence, the proof of Theorem~\ref{thm:classicalbuild} cannot be 
extended for the buildings of exceptional type.

Let again $B$ be an irreducible spherical building of dimension $\geq 2$.
Let $g\in Isom_0(B)$ be a type-preserving isometry with $m$-dimensional fixed point set $Fix(g)$
which is not a subbuilding. 

\begin{lemma}\label{lem:samedecomp}
 Let $g=uk$ with $u$ unipotent be the decomposition
given in Proposition~\ref{prop:parabgp} with respect to an
apartment $A$, that is, $Fix(k)\cap A = s$ is a singular sphere of the same dimension as $Fix(g)$
and $Fix(g)\cap A = Fix(u)\cap s \not\subset \partial Fix(u)$.
Then $g=uk$ is also the decomposition of $g$ with respect to any apartment containing $Fix(k)\cap A$.
\end{lemma}
\begin{proof}
Let $A'$ be an apartment containing $s=Fix(k)\cap A$ and let $g=u'k'$ be the decomposition with
respect to $A'$. Then $Fix(k')\cap A' = s$ and $k^{-1}k'$ must fix $s$.
On the other hand, $u$ and $u'$ fix a neighborhood of $Fix(g)\cap s$ by
Proposition~\ref{prop:parabgp}
and $u^{-1}u'$ is unipotent by Corollary~\ref{cor:fixunipotent}.
It follows that the unipotent isometry $u^{-1}u' = k^{-1}k'$ fixes an apartment and
therefore must be the identity.
\end{proof}

Let $\tau\subset B$ be a $(m-1)$-dimensional face in $\partial Fix(g)$.
Let $ \{\xi\} :=\Si_\tau Fix(g)$. The point $\xi$ is a vertex in  $\Si_\tau B$.
Let $\hat\xi\in\Si_\tau B$ be a vertex antipodal to $\xi$.
Choose an apartment $A\subset B$ 
containing $\tau$ and such that $\xi,\hat\xi\in\Si_\tau A$.
Let $g=uk$ be the decomposition of $g$ with respect to the apartment $A$
given in Proposition~\ref{prop:parabgp}, that is, $u$ is unipotent and $Fix(k)\cap A$ is a 
singular sphere of the same dimension as $Fix(g)$.
The weights for roots in $A$ defined by the unipotent isometry $u$ 
as in Section~\ref{sec:unipotentisom} induce weights for the roots in $\Si_\tau A$
and we obtain a convex function $f_{\Si_\tau (Fix(u) \cap A)}$.
We define $\lambda_{g,\tau,\hat\xi}:= - f_{\Si_\tau (Fix(u) \cap A)} (\xi)$.
Notice that $\lambda_{g,\tau,\hat\xi}$ depends only on $\hat\xi$ and not in the 
chosen apartment $A$, this follows from Lemma~\ref{lem:samedecomp} and 
Theorem~\ref{thm:unipotentisom} applied to the isometry of $\Si_\tau B$ induced by $g$.
We also define the number 
$\lambda_{g,\tau} := \max\{ \lambda_{g,\tau,\hat\xi} \;|\; \hat\xi \text{ antipode of } \xi \text{ in }
\Si_\tau B\}$.

We say that an apartment $A$ {\em supports} $g$ if the following holds:
\begin{enumerate}[(i)]
\item 
$A$ supports the convex subcomplex $Fix(g)$.
\item 
Let $s\subset A$ be the $m$-dimensional singular sphere containing $Fix(g)\cap A$.
If $\tau$ is a $(m-1)$-dimensional face in $\partial(Fix(g)\cap A) \subset \partial Fix(g)$,
then with the notation above holds $\lambda_{g,\tau, \hat\xi} = \lambda_{g,\tau}$, where
$\{\xi\}:=\Si_\tau Fix(g) $ and $\{\xi,\hat\xi\}:=\Si_\tau s $ .
\end{enumerate}

\begin{remark}
Condition (ii) coincides with the
condition about non-reduced roots in the definition of apartments supporting $g$ 
for top-dimensional fixed point sets in
 Section~\ref{sec:topdim}. Hence, this definition generalizes the top-dimensional case.
\end{remark}

\begin{lemma}\label{lem:supportingapt2}
Any two points in $Fix(g)$ are contained in an apartment supporting $g$.
\end{lemma}
\begin{proof}
 By Lemma~\ref{lem:supportingapt}, for any two points in $Fix(g)$
there is an apartment $A$ supporting $Fix(g)$ containing them. 
Let $\tau$ be a $(m-1)$-dimensional face in $\partial(Fix(g)\cap A)\subset \partial Fix(g)$
and let $\{\xi\}=\Si_\tau Fix(g)$.
Let $\hat\xi\in\Si_\tau B$ be an antipode of $\xi$ such that 
$\lambda_{g,\tau, \hat\xi} = \lambda_{g,\tau}$.
Let $h\subset  A$ be the singular hemisphere of dimension $m$
containing $Fix(g)\cap A$ and such that 
$\tau \subset \partial h$. Notice that $\Si_\tau h = \Si_\tau Fix(g)=\{\xi\}$.
There exists an apartment $A'\subset B$ containing $h$ and such that
$\hat\xi\in\Si_\tau A'$. 
Then $A'$ is an apartment supporting $Fix(g)$
and satisfying condition (ii) for the face $\tau$.

Let $s'$ be the singular sphere of dimension $m$ in $A'$ containing $h$.
Suppose that $\tau'$ is  a face in a different codimension one boundary component of 
$\partial(Fix(g)\cap A)$ from $\tau$, and observe that $\Si_{\tau'} h =\Si_{\tau'} s'$.
Therefore if $A$ already satisfied condition (ii) for $\tau'$, then $A'$
still satisfied condition (ii) for $\tau'$.
We can repeat the construction until we get an apartment supporting $Fix(g)$
satisfying condition (ii) for all codimension one faces.
\end{proof}

\begin{theorem}\label{thm:classicalbuild}
Let $B$ be an irreducible spherical building of dimension at least $2$ 
and not of type $F_4, E_6,E_7,E_8$.
Let $g\in Isom_0(B)$ be an isometry such that $Fix(g)\subset B$ is not a subbuilding.
Let $x\in Fix(g)$ and let $A\subset B$ be an apartment
containing $x$ and supporting $g$. 
Let $u$ be the unipotent part of $g$ with respect to $A$ as in Proposition~\ref{prop:parabgp}.
Then the function $f(x) := f_{Fix(u)\cap A} (x)$ as defined in Section~\ref{sec:unipotentisom}
 does not depend on the choice of the apartment $A$ supporting $g$. 
In particular, $f$ defines a nicely convex function
in $Fix(g)$ and it has a unique minimum $x_0\in Fix(g)$, the {\em weighted incenter} of 
$Fix(g)$. Moreover, $\rad(Fix(g),x_0)\leq \pihalf$. 
\end{theorem}

\begin{proof}
Let $A'$ be another apartment supporting $g$ and containing $x\in Fix(g)$.
Write $g=uk$ with respect to $A$ and $g=u'k'$ with respect to $A'$.
Let $s=Fix(k)\cap A$ and $s'=Fix(k')\cap A'$.

As in the top-dimensional case, we want to show that $Fix(u)\cap s' \subset Fix(k)\cap A'=Fix(k)\cap s'$.
The proof is the same, using the fact that $A'$ supports $Fix(g)$:
Suppose not, then there is a face $\tau$ of full dimension in $\partial(Fix(k)\cap s')$
and two faces $\nu_0, \nu_1\subset  Fix(u)\cap s'$ of full dimension in $s'$ 
such that $\nu_0\cap\nu_1=\tau$,
$\nu_0\not\subset Fix(k)\cap s'$ and $\nu_1\subset Fix(k)\cap Fix(u)$.
This implies that $u$ fixes $St_\tau B$ pointwise. 
It follows,
$g\nu_0= u(k\nu_0) = k\nu_0 \neq \nu_0$, hence, $\nu_0\notin Fix(g)$.
Therefore $\tau\subset \partial(Fix(g)\cap A')\subset \partial Fix(g)$ because $A'$ supports $Fix(g)$.
On the other hand, $Fix(k)\subset B$ is a subbuilding and therefore there exists a chamber
$\nu'\subset Fix(k)$ such that $\nu'\cap\nu_1=\tau$.
But $u$ must also fix $\nu'$, hence, $g\nu'=uh\nu'=\nu'$ and $\nu'\in Fix(g)$.
This contradicts the fact that $\tau\subset \partial Fix(g)$.
So we conclude that $Fix(u)\cap s' \subset Fix(k)\cap s'$ and in particular
$Fix(u)\cap s' \subset Fix(g)\cap A'=Fix(u')\cap s'$.

Let us consider first the case when $x\in Fix(g)$ is a vertex.
Suppose that $f_{Fix(u')\cap A'}(x) >  f_{Fix(u)\cap A}(x)$.
Let $f_{Fix(u)\cap A'}$ be the function in $Fix(u)\cap A'$ with the weights defined by $u$.
Then $f_{Fix(u)\cap A}(x) = f_{Fix(u)\cap A'}(x)$ by Theorem~\ref{thm:unipotentisom}.
Now we apply Lemma~\ref{lem:vertexposcodim} to $K_1=Fix(u')\cap A'$ and $K_2=Fix(u)\cap A'$,
let $\alpha\in\Lambda_{K_1}^{min}$ be a root such that 
$f_{K_1}(x) = -\mu_{\alpha}\sin d(x,\partial\alpha)$,
then the Lemma implies that 
$h:=\alpha\cap s' \in \Lambda_{K_1\cap s'}^{min}\cap \Lambda_{K_2\cap s'}^{min}$.

Notice that $\alpha\in \Lambda_{CH(K_1,\partial h)}$ and therefore
$f_{K_1}(x) \geq f_{CH(K_1,\partial h)}\geq -\mu_{\alpha}\sin d(x,\partial\alpha)$.
It follows that $f_{K_1}(x) = f_{CH(K_1,\partial h)}(x)$.
Let $\{\xi\}:= \Si_{\partial h}(K_i\cap s') = \Si_{\partial h}(Fix(g)\cap A') = \Si_{\partial h} Fix(g)$.
By Lemma~\ref{lem:inductionarg} we have
$f_{K_1}(x) =  f_{CH(K_1,\partial h)}(x) = \sin d(x,\partial h)f_{\Si_{\partial h} K_1}(\xi)$
and 
$f_{K_2}(x) \geq  f_{CH(K_2,\partial h)}(x) = \sin d(x,\partial h)f_{\Si_{\partial h} K_2}(\xi)$.
Our assumption $f_{K_1}(x) >  f_{K_2}(x)$
implies $f_{\Si_{\partial h} K_1}(\xi) > f_{\Si_{\partial h} K_2}(\xi)$.

Let $\tau$ be a face in $\partial h \cap (K_1\cap s') \cap (K_2\cap s')\subset \partial Fix(g)$.
Let $\xi_1\in\Si_\tau B$ be the vertex such that $\Si_\tau s' = \{\xi,\xi_1\}$.
Then $-\lambda_{g,\tau,\xi_1} = f_{\Si_\tau (Fix(u')\cap A')} (\xi) = f_{\Si_{\partial h} K_1}(\xi)$.
Since $K_2\cap s' = Fix(u)\cap s' \subset Fix(k)\cap A'$ and $Fix(k)$ is a subbuilding,
we find an apartment $A''\subset Fix(k)$ containing $\tau$ and $\xi\in\Si_\tau A''$.
Let $\xi_2\in \Si_\tau A''$ be the antipode of $\xi$.
Then 
$-\lambda_{g,\tau,\xi_2} = f_{\Si_\tau (Fix(u)\cap A'')} (\xi) = 
f_{\Si_\tau (Fix(u)\cap A')} (\xi)= f_{\Si_{\partial h} K_2}(\xi)$
(the second equality follows from Theorem~\ref{thm:unipotentisom}).
It follows that $\lambda_{g,\tau,\xi_2} > \lambda_{g,\tau,\xi_1}$.
Since $A'$ supports $g$, condition (ii) implies that $\lambda_{g,\tau,\xi_1} = \lambda_{g,\tau}$.
Hence $\lambda_{g,\tau,\xi_2} >\lambda_{g,\tau} $, contradicting the definition
of $\lambda_{g,\tau} $.
Thus, $f_{Fix(u')\cap A'}(x) \leq  f_{Fix(u)\cap A}(x)$ and interchanging the roles of $A$ and $A'$,
we obtain $f_{Fix(u')\cap A'}(x) =  f_{Fix(u)\cap A}(x)$.
We have shown the theorem in the case when $x\in Fix(g)$ is a vertex.

Now we proceed as in the proof of Theorem~\ref{thm:Dn} to show the general case.
Let $x\in Fix(g)$ and let $\tau\subset Fix(g)$ be the face containing $x$ in its interior.
Let $v_{i_1},\dots, v_{i_k}$ be the vertices of $\tau$ with $v_{i_j}$ of type $i_j$
and $i_1<\dots<i_k$. After identifying $\tau$ with a subset of the unit
round sphere in $\R^k$, we can write $x=\sum_{j=1}^k a_j v_{i_j}$ with $a_j\geq 0$.
Suppose again that $f_{Fix(u')\cap A'}(x) >  f_{Fix(u)\cap A}(x)$.

Let $\alpha\in \Lambda_{Fix(u')\cap A'}$ be a root such that 
$f_{Fix(u')\cap A'} (x) = -\mu_\alpha \sin d(x,\partial\alpha)$.
Suppose first that there is a vertex $v=v_{i_j}$ of $\tau$ with $v\in \partial\alpha$.
Then by Lemma~\ref{lem:inductionarg}, we have
$f_{Fix(u')\cap A'}=\sin d(x,v)f_{\Si_v(Fix(u')\cap A')}(\ora{vx})$.
The link $\Si_v B$ has again no factors of exceptional type, 
then using induction on the rank of the building applied to the isometry
of $\Si_v B$ induced by $g$, we obtain
$f_{\Si_v(Fix(u')\cap A')}(\ora{vx}) = f_{\Si_v(Fix(u)\cap A)}(\ora{vx})$.
It follows again by Lemma~\ref{lem:inductionarg} that
$f_{Fix(u)\cap A}(x) \geq \sin d(x,v)f_{\Si_v(Fix(u)\cap A)}(\ora{vx}) =
\sin d(x,v)f_{\Si_v(Fix(u')\cap A')}(\ora{vx}) = f_{Fix(u')\cap A'}(x)$, 
contradicting our assumption $f_{Fix(u)\cap A} (x) < f_{Fix(u')\cap A'} (x)$.
Hence, $\tau$ is contained in the interior of $\alpha$.

Recall that by Lemmata~\ref{lem:Bn} and \ref{lem:Dn}, 
$\mu_\alpha\sin d(v_{i_j},\partial\alpha)$ can only take at most the 
two values $\lambda_{i_j}, 2\lambda_{i_j}$.
Let $r$ be the smallest number such that $\mu_\alpha\sin d(v_{i_r},\partial\alpha)=\lambda_{i_r}$.
Then by Lemmata~\ref{lem:Bn2} and \ref{lem:Dn2}, 
$\mu_\alpha\sin d(v_{i_j},\partial\alpha)=\lambda_{i_j}$ for all $j\geq r$.
Therefore
\begin{align*}
 f_{Fix(u')\cap A'}(x) &= -\mu_\alpha\sin d(x,\partial\alpha)= 
-\sum_{j=1}^k a_j\mu_\alpha\sin d(v_{i_j},\partial\alpha) 
=-2\sum_{j=1}^{r-1}a_j\lambda_{i_j}  -\sum_{i=r}^ka_j\lambda_{i_j} \\
& > f_{Fix(u)\cap A}(x) = \max_{\beta\in\Lambda_{Fix(u)\cap A}}\{-\mu_\beta\sin d(x,\partial\beta)\} \\
&=
\max_{\beta\in\Lambda_{Fix(u)\cap A}}\{-\sum_{j=1}^k a_j\mu_\beta\sin d(v_{i_j},\partial\beta)  \}.
\end{align*}
It follows that for each $\beta\in\Lambda_{Fix(u)\cap A}$, there must be a $j\geq r$ such that
$\mu_\beta\sin d(v_{i_j},\partial\beta) = 2\lambda_{i_j}$.
Then again by Lemma~\ref{lem:Dn2}, $\mu_\beta\sin d(v_{i_r},\partial\beta) = 2\lambda_{i_r}$
for all $\beta\in\Lambda_{Fix(u)\cap A}$. 
Thus, $f_{Fix(u)\cap A}(v_{i_r}) = -2\lambda_{i_r} < -\lambda_{i_r} = 
-\mu_\alpha\sin d(v_{i_j},\partial\alpha) \leq f_{Fix(u')\cap A'}(v_{i_r})$, 
contradicting the case for vertices in $Fix(g)$.
We conclude in the general case that $f_{Fix(u')\cap A'}(x) =  f_{Fix(u)\cap A}(x)$.
\end{proof}

\section{Some special cases and applications}\label{sec:applications}

\subsection{Long root subgroups}

Recall from Theorem~\ref{thm:commutator}, that if $\alpha$ is a non-reduced root,
 then we have some
flexibility in defining the root group $U_{2\alpha} \subset U_\alpha$.
The next proposition implies that there is a unique maximal such subgroup and it coincides with
the pointwise stabilizer of the ball of radius $\pihalf$ containing $\alpha\subset B$.
This gives a geometric definition of $U_{2\alpha}$.

\begin{proposition}\label{prop:pihalfball}
Let $B$ be an irreducible spherical building of dimension at least two,
a Moufang generalized triangle or a Moufang generalized quadrangle
with associated root system $\Phi$.
Let $\alpha\subset B$ be a root and let $x_\alpha$ be its center.

\noindent If $\Phi$ is reduced and $\alpha\in\Phi$ is a long root, then 
$U_\alpha = Fix_{Isom(B)}(B_{\pihalf}(x_\alpha))$.

\noindent If $\Phi$ is non-reduced and $\alpha\in\Phi$ is a non-reduced root, then
$U_{2\alpha} \subset Fix_{Isom(B)}(B_{\pihalf}(x_\alpha)) \subset U_\alpha$. 
Moreover, we can replace $U_{2\alpha}$ with $Fix_{Isom(B)}(B_{\pihalf}(x_\alpha))$
and Theorem~\ref{thm:commutator} remains valid.
\end{proposition}
\begin{proof}
 By the definition of root subgroup, 
it is clear that $Fix_{Isom(B)}(B_{\pihalf}(x_\alpha)) \subset U_\alpha$
for any root $\alpha$.
Let now $1\neq g\in U_\alpha$ (or $U_{2\alpha}$ in the non-reduced case)
 as in the statement of the proposition.
Let $f$ be the function on $Fix(g)$ given by Theorem~\ref{thm:unipotentisom}.
Let $A$ be an apartment containing the root $\alpha$.
Then $f(x_\alpha) = f_{Fix(g)\cap A }(x_\alpha) = f_\alpha (x_\alpha) = -\mu_\alpha$.
By the hypothesis of the proposition $\mu_\alpha$ is the norm of the longest root in $\Phi$.
Now for any other apartment $A'$ containing $x_\alpha$ we have
$ -\mu_\alpha = f(x_\alpha) = f_{Fix(g)\cap A' }(x_\alpha) =
\max_{\beta\in\Lambda_{Fix(g)\cap A'}}\{-\mu_\beta \sin d(x_\alpha,\partial\beta)\}
 \geq \max_{\beta\in\Lambda_{Fix(g)\cap A'}}\{-\mu_\beta\} \geq -\mu_\alpha$.
The equality implies that $x_\alpha$ must be the center of every root in $\Lambda_{Fix(g)\cap A'}$.
That is, $Fix(g)\cap A' = B_{\pihalf}(x_\alpha) \cap A'$. It follows that $Fix(g)=B_{\pihalf}(x_\alpha)$.

Now we prove the second assertion. For a non-reduced root $\alpha$, let 
$\bar U_{2\alpha}:=Fix_{Isom(B)}(B_{\pihalf}(x_\alpha))$. We have to verify the commutator relations
for these subgroups. Let $\beta$ be another root and let $x_\beta$ be its center.
Then $d(x_\alpha,x_\beta)\in\{0,\frac{\pi}{4},\pihalf,\frac{3\pi}{4},\pi\}$.

Case $d(x_\alpha,x_\beta)=\frac{\pi}{4}$: 
Then $[\bar U_{2\alpha},U_\beta]\subset [U_\alpha,U_\beta]=1$.

Case $d(x_\alpha,x_\beta)=\pihalf$: Let $g\in\bar U_{2\alpha}$ and $h\in U_\beta$.
Then $h$ fixes $x_\alpha$ and therefore stabilizes $B_{\pihalf}(x_\alpha)$. 
Since $g$ fixes $B_{\pihalf}(x_\alpha)$ pointwise, 
it follows that $[g,h]$ also fixes $B_{\pihalf}(x_\alpha)$ pointwise.
Similarly, $g$ fixes $x_\beta$ and therefore stabilizes $St_{x_\beta}B$. 
On the other hand $h$ fixes $St_{x_\beta}B$ pointwise and
it follows that $[g,h]$ also fixes $St_{x_\beta}B$ pointwise.
We conclude that $Fix([g,h])\supset B_{\pihalf}(x_\alpha)\cup St_{x_\beta}B$.
But the convex hull of $B_{\pihalf}(x_\alpha)\cup St_{x_\beta}B$ is $B$, hence $Fix([g,h])=B$
and $[g,h]=1$. We obtain $[\bar U_{2\alpha},U_\beta]=1$.

Case $d(x_\alpha,x_\beta)=\frac{3\pi}{4}$: 
Let $g\in\bar U_{2\alpha}$ and $h\in U_\beta$. Then $[h,g]=g_1g_2$ with
$g_i\in U_{\beta+i\alpha}$ (see e.g.\ \cite[Prop.\ 11.17]{Weiss:sphbuild}).
By \cite[Lemma 2.1]{Tits:rootdata}, $g_1$ is conjugated to $g$ and therefore
it fixes a ball of radius $\pihalf$. This implies that $g_1\in \bar U_{2(\beta+\alpha)}$. 
Hence $[\bar U_{2\alpha},U_\beta]\subset \langle \bar U_{2(\beta+\alpha)}, U_{\beta+2\alpha}\rangle$.

Case $d(x_\alpha,x_\beta)=0$: Let $g\in\bar U_{2\alpha}$ and $h\in U_\alpha$.
Choose a root $\gamma$ with $d(x_\alpha,x_\gamma)=\frac{\pi}{4}$ and let $u\in U_\gamma$.
Then by \cite[Lemma 2.1]{Tits:rootdata}, there is a $g'\in U_{\alpha-\gamma}$ 
conjugate to $g$ (and therefore $g'\in \bar U_{2(\alpha-\gamma)}$)
and a $k\in U_{2\alpha-\gamma}$ such that $[u,g']=gk$.
Since $[U_\gamma,U_\alpha]=1$ and $[U_{2\alpha-\gamma},U_\alpha]=1$, 
the elements $u$ and $k$ commute with $h$. Moreover, by the second case above
$[\bar U_{2(\alpha-\gamma)},U_\alpha]=1$ and we also see that $g'$ commutes with $h$.
It follows that $[g,h]=1$ and therefore $[\bar U_{2\alpha},U_\alpha]=1$, in other words,
$\bar U_{2\alpha} \subset Z(U_\alpha)$.
\end{proof}

\subsection{Commuting unipotent elements}

The following result is well known in the setting of algebraic groups. 
We give here a geometric proof that works for any spherical building. 

\begin{proposition}\label{prop:commutingunipotent}
 The product of two commuting unipotent isometries is again unipotent.
\end{proposition}
\begin{proof}
Let $g_1,g_2$ be two commuting unipotent isometries. 
Then $g_i$ stabilizes $Fix(g_{3-i})$.
Let $f_i$ be the function on $Fix(g_i)$
given by Theorem~\ref{thm:unipotentisom} and let $x_i$ be the corresponding weighted incenter
of $Fix(g_i)$.

Let $\{i,j\}=\{1,2\}$.
Let $A$ be an apartment containing $x\in Fix(g_i)$.
Then $f_i(g_{j}x)= f_{Fix(g_i)\cap g_{j}A}(g_{j}x) = f_{Fix(g_j^{-1}g_ig_j)\cap A}(x)=
 f_{Fix(g_i)\cap A}(x) = f_i(x)$. Hence, the function $f_i$ is $g_j$-invariant. Therefore
$g_j$ must fix the unique minimum $x_i$ and $x_1x_2\subset Fix(g_1)\cap Fix(g_2)$.
Notice that $x_i$ is an interior point of $Fix(g_i)$ by definition. It follows that the midpoint
$x_0$ of the segment $x_1x_2$ is interior in $Fix(g_1)$ and in $Fix(g_2)$.
This implies that there is a chamber $\sigma\subset Fix(g_1)\cap Fix(g_2)$ containing $x_0$.
Then, by Corollary~\ref{cor:fixunipotent}, $g_i\in U_\sigma$ and in particular 
$g_1g_2\in U_\sigma$.
\end{proof}

\subsection{Jordan decomposition}\label{sec:jordan}

In this section we want to consider special examples of spherical buildings that include the 
buildings associated to algebraic groups and isometries for which there exists a Jordan decomposition.
First we will explain the setting that occurs in the algebraic groups and then we state the results in
a purely geometric manner forgetting the algebraic group structure.

Let $G$ be a semisimple algebraic group defined over an algebraically closed field $k$.
Let $B_{G,k}$ denote the associated spherical building. 
The faces of $B_{G,k}$ correspond to parabolic subgroups and the chambers to minimal
parabolic subgroups, that is, Borel subgroups.
An apartment corresponds to a maximal torus, the faces contained in the apartment are the parabolics
containing the maximal torus.
The group $G(k)$ acts on $B_{G,k}$ by type-preserving isometries. 
An element $g\in G(k)$ has fixed point set a subbuilding if and only if it is a semisimple element.
Let $s\in G(k)$ be semisimple, the fixed point set $Fix(s)$ consists on all apartments corresponding
to maximal tori containing $s$.
There is a torus $S\subset G$ such that $C_G(S)=C_G(s)$. The torus $S$ corresponds to
a singular sphere $\varsigma\subset B_{G,k}$ and $Fix(s)$ is the union of the apartments 
containing $\varsigma$.
In particular $Fix(s)$ is always top-dimensional.

Let now $G$ be a semisimple algebraic group defined over an arbitrary field $k$. 
Suppose further that $G$ is $k$-split, that is, $G$ contains a $k$-split 
(diagonalizable over $k$) maximal torus. Let $\bar k$ be the algebraic closure of $k$.
We have that $B_{G,k}$ is a subbuilding of $B_{G,\bar k}$. 
Since $G$ is $k$-split, $B_{G,k}$ is top-dimensional in $B_{G,\bar k}$. 
Let $s\in G(k)$ be a semisimple element. 
In general, $Fix(s)\subset B_{G,k}$ is not top-dimensional anymore.
Suppose first that $Fix(s)\subset B_{G,k}$ is top-dimensional. 
Then by the above discussion, the apartments in the fixed point set of the action
$s\acts B_{G,\bar k}$ must contain the singular sphere 
$\varsigma\subset B_{G,\bar k}$. This implies that the torus $S$ is $k$-defined and
$\varsigma\subset B_{G,k}$. The fixed point set $Fix(s)$ is the union of the apartments in $
B_{G, k}$ containing $\varsigma$.

Suppose now that the field $k$ is perfect. Let $g\in G(k)$, then $g$ has a unique
Jordan decomposition $g=su$ with $s\in G(k)$ semisimple, $u\in G(k)$ unipotent and $s,u$ 
commute with each other. 
In the case of algebraically closed fields, this implies that
$Fix(g)$ is always top-dimensional and therefore it is a subbuilding or has circumradius
$\leq \pihalf$ by Theorem~\ref{thm:topdim}. 
We generalize this conclusion to perfect fields and non-split groups 
in Propositions~\ref{prop:fixsetjordan} and \ref{prop:fixsetnonsplit}.

Motivated by this discussion, we make the following definition. We say that an isometry
$g\in Isom(B)$ of a spherical building $B$ is {\em split} if $Fix(s)$ is the union of all apartments
containing a singular sphere $\varsigma\subset B$.
In particular, the fixed point set of a split isometry 
factorizes as a spherical join $Fix(k)\cong \Si_{\varsigma} B\circ \varsigma$.

In the following results we consider isometries $g$ of spherical buildings
that admit a kind of {\em Jordan decomposition}, that is, can be written as $g=uk$ with
$u$ unipotent, $Fix(k)$ is a subbuilding and $u$ and $k$ commute.

\begin{lemma}\label{lem:fixuk}
Let $B$ be an irreducible spherical building.
Let $k$ be an isometry, whose fixed point set is a subbuilding, and let
$u$ be a unipotent isometry. Suppose that $u$ and $k$ commute.
Then $Fix(u)\cap Fix(k)$ is a top-dimensional 
subcomplex of the building $B'=Fix(k)$ with its thick structure.
\end{lemma}
\begin{proof}
Let $x_0$ be the weighted incenter of $Fix(u)$ given by Theorem~\ref{thm:unipotentisom}. 
Then $k$ fixes $x_0$ because it commutes with $u$. 
By definition $x_0$ must be an interior point of $Fix(u)$, therefore by Corollary~\ref{cor:star},
$St_{x_0} B\subset Fix(u)$.
Since $Fix(k)$ is a subbuilding, it follows that $x_0$ is an interior point of $Fix(u)\cap Fix(k)$
and $Fix(u)\cap Fix(k)$ is of full dimension in $Fix(k)$.

Let $\tau\subset \partial(Fix(u)\cap Fix(k))$ be a face of full dimension and let 
$\sigma,\sigma'\subset Fix(k)$ be faces of full dimension containing $\tau$ and such that 
$\sigma\subset Fix(u)$, $\sigma'\not\subset Fix(u)$. 
Observe that $u$ stabilizes $Fix(k)$ because $u$ and $k$ commute.
Then $\sigma,\sigma',u\sigma'$ are three
pairwise distinct faces in $Fix(k)$ containing $\tau$. It follows that $\tau$ is contained in a wall
with respect to the thick structure of $Fix(k)$.
\end{proof}

Notice that if $k$ is a split isometry with $Fix(k)\cong \Si_{\varsigma} B\circ \varsigma$,
then a top-dimensional subcomplex $K\subset Fix(k)\subset B$ is a subcomplex
with respect to the thick structure of $Fix(k)$ if and only if $K$ contains the singular
sphere $\varsigma$.
In the next result we see that the converse of Lemma~\ref{lem:fixuk} is also true if 
the isometry $k$ is split.

\begin{lemma}
Let $B$ be an irreducible spherical building.
Let $k$ be a split isometry with $Fix(k)\cong \Si_{\varsigma} B\circ \varsigma$.
If the fixed point set of a unipotent isometry $u$ contains the singular sphere $\varsigma$,
then $u$ and $k$ commute. 
\end{lemma}
\begin{proof}
 Let $A\subset B$ be an apartment containing $\varsigma$ and such that $K=Fix(u)\cap A$ is
top-dimensional. Then by Proposition~\ref{prop:coordunipotent}, $g$ is a product of root elements
of roots in $\Lambda_K = \{\alpha\subset A\;|\; K\subset A\}\subset 
\{\alpha\subset A\;|\; \varsigma\subset A\}$. Hence, we may assume that $u\in U_\alpha$
for some $\alpha\subset A$ containing $\varsigma$.
Then $kuk^{-1}\in U_\alpha$ because $k$ fixes $\alpha$.
Further, the action of $kuk^{-1}$ and $u$ on $\Si_\varsigma B$ is the same because $k$ acts
as the identity on $\Si_\varsigma B$. Since $B$ is irreducible, this implies that 
$kuk^{-1}=u$.
\end{proof}

The following proposition applies in particular to the buildings
associated to algebraic groups $G$ defined over algebraically closed fields $k$
and isometries $g\in G(k)$.

\begin{proposition}\label{prop:fixsetjordan}
Let $B$ be an irreducible spherical building.
Let $k$ be a split isometry and let $u$ be a unipotent isometry such that 
$u$ and $k$  commute. Let $g=uk$. Then $Fix(g)=Fix(u)\cap Fix(k)$.
\end{proposition}
\begin{proof}
Clearly, $Fix(g)\supset Fix(u)\cap Fix(k)$.
Let $\sigma\subset B$ be a chamber such that 
$\sigma \cap (Fix(u)\cap Fix(k))$ is a panel $\tau \subset \partial(Fix(u)\cap Fix(k))$.
By Lemma~\ref{lem:fixuk}, $\tau$ is contained in a wall of the thick structure of 
$Fix(k)\cong \Si_{\varsigma} B\circ \varsigma$.
This implies that there is a panel $\nu\subset \Si_{\varsigma} B$ such that
$\tau \subset \nu\circ \varsigma \subset Fix(k)\cong \Si_{\varsigma} B\circ \varsigma$.
In particular, $St_\tau B\subset St_\nu(\Si_\varsigma B)\circ\varsigma \subset Fix(k)$. 
This in turn implies that $\sigma\subset Fix(k)$.
Since, $\sigma\not\subset Fix(u)\cap Fix(k)$, 
it follows that $\sigma\not\subset Fix(g)$.
Therefore, we conclude that $Fix(g)=Fix(u)\cap Fix(k)$.
\end{proof}

Let again $G$ be a semisimple algebraic group defined over a field $k$,
but we do not assume that $G$ is $k$-split.
Let $s\in G(k)$ be a semisimple element. 
Even if $Fix(s)\subset B_{G,k}$ is top-dimensional, it may not be top-dimensional in $B_{G,\bar k}$. 
Thus, the isometry $s$ of $B_{G,k}$ must not be split. 
Nevertheless, we can use Proposition~\ref{prop:fixsetjordan} applied to $B_{G,\bar k}$ to conclude 
that for a perfect field $k$ (that is, when we have a Jordan decomposition), the fixed point set 
$Fix(g)$ is either a subbuilding or it has circumradius $\leq \pihalf$.

\begin{proposition}\label{prop:fixsetnonsplit}
Let $\tilde B$ be an irreducible spherical building
Let $\tilde k$ be a split isometry 
and let $\tilde u$ be a unipotent isometry such that 
$\tilde u$ and $\tilde k$  commute. 
Suppose further that $\tilde u,\tilde k$ stabilize 
a subbuilding $B\subset \tilde B$, in particular, their restrictions to $B$
induce isometries $u,k\in Isom(B)$. Assume also that $u$ is unipotent. Let $g=uk$. 
Then $Fix(g)=Fix(k)\cap Fix(u)\subset B$ 
is either a subbuilding or it has circumradius $\leq \pihalf$.
\end{proposition}
\begin{proof}
 By Proposition~\ref{prop:fixsetjordan}, 
$Fix(g)=B\cap Fix(\tilde u\tilde k) = B\cap Fix(\tilde k)\cap Fix(\tilde u)=Fix(k)\cap Fix(u)$.
Theorem~\ref{thm:unipotentisom} implies that $Fix(u)$
has a unique weighted incenter $x_0$. 
Since $u$ and $k$ commute, $k$ fixes $x_0$ and therefore
$x_0\in Fix(g)$. It follows that 
$\rad(Fix(g))\leq \rad(Fix(g),x_0)\leq \rad(Fix(u),x_0)\leq\pihalf$ 
again by Theorem~\ref{thm:unipotentisom}.
\end{proof}

\begin{remark}
Proposition~\ref{prop:fixsetnonsplit} in the case of semisimple Lie groups 
gives another proof of \cite[Proposition 4.1.5]{Eberlein:nonpositive}, see also
\cite[Lemma 12.3]{Mostow:rigidity}.
\end{remark}

Suppose we are in a setting where a Jordan decomposition always exists. 
It is a natural question to ask for a geometric way of finding this decomposition.\
This is what \cite[Problem 2.19.11]{Eberlein:nonpositive} is about in the case of
symmetric spaces of noncompact type.
Let us rephrase the statement of the \cite[Conjecture 2.19.11]{Eberlein:nonpositive} 
in our notation. Let $G$ be a noncompact semisimple Lie group and let $B$ be its associated
spherical building. $B$ is the Tits boundary of the symmetric space $X=G/K$,
where $K$ is a maximal compact subgroup. Let $g\in G$ be a parabolic isometry. Then
the fixed point set $Fix(g)$ of $g$ in $B$ has circumradius $\leq \pihalf$ (cf.\ Section~\ref{sec:intro}).
Let $x_0$ be the unique circumcenter of the set of circumcenters of $Fix(g)$
and let $\tau\subset B$ be the face containing $x_0$ in its interior.
Let $A\subset B$ be an apartment containing $\tau$. 
Let $g=uk$ be the decomposition given by Proposition~\ref{prop:parabgp} with respect to the 
apartment $A$, where $u\in U_\tau$ and $k\in L_\tau H$.
Then the conjecture asks if $k$ is semisimple and $g=uk$ is the Jordan decomposition of $g$.
As stated the conjecture cannot be true as we can see in the following example.

\begin{example}
 Let 
$g=\left(\begin{smallmatrix}
a&&1\\ &a^{-2}& \\ &&a
\end{smallmatrix} \right) \in SL(3,\R)$ with $a\neq 1$. Then the fixed point set $Fix(g)$
in $B=\tits (SL(3,\R)/SO(3))$ is a root $\alpha\subset B$. 
The unipotent part $u$ in the Jordan decomposition of $g=us$ 
is an element in the root group $U_\alpha$.
The center of $\alpha$ is the center
of a chamber $\sigma$. Let $A$ be an apartment such that $A\cap \alpha = \sigma$.
Then the decomposition $g=u'k'$ with respect to this apartment cannot be the Jordan decomposition
because $Fix(u')\cap A = \sigma$ and $u'$ cannot be a root element.
On the other hand, the decomposition with respect to any apartment containing $\alpha$
is the Jordan decomposition.
\end{example}

In the example above, the key to obtain the Jordan decomposition
was to choose an apartment supporting the fixed point set $Fix(g)$. 
This works in general for a split algebraic group $G$ over a perfect field and $g\in G$ with 
top-dimensional fixed point set. This is the assertion of the following proposition,
thus giving a solution of \cite[Problem 2.19.11]{Eberlein:nonpositive} in this case.
Notice that it is actually not important whether the apartment contains the circumcenter of
the fixed point set or not.

\begin{proposition}
Let $B$ be an irreducible spherical building.
Let $g$ be an isometry admitting a decomposition $g=uk$ with $u$ unipotent,
$k$ split and such that $u$ and $k$  commute. 
Then $g=uk$ is the decomposition as in Proposition~\ref{prop:parabgp} with respect 
to any apartment supporting $Fix(g)$.
\end{proposition}
\begin{proof}
 By Lemma~\ref{lem:fixuk},
 $Fix(g)$ is a subcomplex of $Fix(k)\cong \Si_{\varsigma} B\circ \varsigma$ with respect to its 
thick structure. It follows that any apartment $A$ supporting $Fix(g)$ must contain the singular
sphere $\varsigma$. This in turn implies that $A\subset Fix(k)$. 
Let $g=u'k'$ be the decomposition with respect to $A$.
Then $u^{-1}u'=kk'^{-1}$ fixes $A$. 
By Proposition~\ref{prop:fixsetjordan} we have
$Fix(g)=Fix(u)\cap Fix(k)$, hence,
$Fix(u)\cap A = Fix(g)\cap A = Fix(u')\cap A$ and it follows that $u^{-1}u'=kk'^{-1}$ is a unipotent isometry
by Corollary~\ref{cor:fixunipotent}. Then $u^{-1}u'=kk'^{-1}$ must be the identity.
\end{proof}

Taking an apartment supporting $Fix(g)$ does not work anymore in the general case. 
Actually, it is not
possible to extract the Jordan decomposition of $g$ just by considering its fixed point set as the
following example shows. 

\begin{example}
Let 
$g=\left(\begin{smallmatrix}
R&Id_2\\ &R \end{smallmatrix} \right) \in SL(4,\R)$, 
where $Id_2\in SL(2,\R)$ is the identity matrix and
$\pm Id_2\neq R\in SO(2)$ is a rotation.   
The fixed point set of $g$ consists of only one point.
It is the vertex of $B=\tits (SL(4,\R)/SO(4))$ 
corresponding to the plane $\langle e_1,e_2\rangle \subset \R^4$. 
\end{example}


\bibliography{../../MyBibliography}
\bibliographystyle{amsalpha}

\noindent {\small
\textsc{Mathematisches Institut, Universit\"at M\"unchen, Theresienstr. 39, 
D-80333 M\"unchen, Germany}\\
{\em E-mail:} {\texttt cramos@mathematik.uni-muenchen.de}}

\end{document}